\documentclass[11pt]{article}
\usepackage[top=27mm, bottom=27mm, left=22mm, right=22mm]{geometry}
\usepackage[square, comma, sort&compress, numbers]{natbib}
\usepackage{hyperref}
\usepackage{amsfonts}
\usepackage{mathrsfs}
\usepackage{comment}
\usepackage{amsmath}
\usepackage{amssymb}
\usepackage{amsthm}
\usepackage{amscd}
\usepackage{graphicx}
\usepackage{indentfirst}
\usepackage[all]{xy}
\usepackage{titlesec}
\usepackage{enumerate}
\usepackage{bm}
\usepackage{relsize}
\usepackage{enumitem}
\usepackage{color}
\usepackage{dsfont}
\usepackage{arydshln}
\usepackage[capitalize]{cleveref}
\usepackage{appendix}
\usepackage[normalem]{ulem}
\newtheorem{theorem}{Theorem}[section]
\newtheorem{lemma}[theorem]{Lemma}

\newtheorem{corollary}[theorem]{Corollary}
\newtheorem{conjecture}[theorem]{Conjecture}
\newtheorem{definition}[theorem]{Definition}
\newtheorem{construction}[theorem]{Construction}

\newtheorem{claim}[theorem]{Claim}

\newcommand{\ma}{\mathcal}

\newcommand{\mr}{\mathscr}

\newcommand{\fr}{\frac}
\newcommand{\lc}{\lceil}
\newcommand{\rc}{\rceil}

\newcommand{\ep}{\epsilon}

\newcommand{\e}{{\rm{ex}}}

\begin{document}

\title{Sharp asymptotic bounds for uniform union-free hypergraphs}

\date{}

\makeatletter
\def\thanks#1{\protected@xdef\@thanks{\@thanks
        \protect\footnotetext{#1}}}
\makeatother

\author{Miao Liu\thanks{M. Liu is with the Research Center for Mathematics and Interdisciplinary Sciences, Shandong University, Qingdao, China, and Extremal Combinatorics and Probability Group (ECOPRO), Institute for Basic Science (IBS), Daejeon, South Korea. Email: liumiao10300403@163.com.}, Chong Shangguan\thanks{C. Shangguan is with the Research Center for Mathematics and Interdisciplinary Sciences, Shandong University, Qingdao 266237, China, and the Frontiers Science Center for Nonlinear Expectations, Ministry of Education, Qingdao 266237, China. Email: theoreming@163.com.}, 
and Chenyang Zhang \thanks{C. Zhang is with the Research Center for Mathematics and Interdisciplinary Sciences, Shandong University, Qingdao 266237, China. Email: lxzcyang@163.com.}
}

\maketitle

\begin{abstract}
    \noindent An $r$-uniform hypergraph is called $t$-union-free if any two distinct subsets of at most $t$ edges have distinct union. The study of union-free hypergraphs has multiple origins and a long history, dating back to the works of Kautz and Singleton (1964) in coding theory, Bollob\'as and Erd\H{o}s (1976) in combinatorics, and Hwang and S\'os (1987) in group testing. Let $U_t(n,r)$ denote the maximum number of edges in an $n$-vertex $t$-union-free $r$-uniform hypergraph.  In this paper, we determine the asymptotic behavior of $U_t(n,r)$, up to a lower order term, for almost all $t\ge 3$ and $r\ge 3$. This significantly advances the understanding of this extremal function, as previously, only the asymptotics of $U_2(n,3)$ and $U_2(n,4)$ were known. As a key ingredient of our proof, we establish the existence of near-optimal locally sparse induced hypergraph packings, which is of independent interest.
\end{abstract}

\section{Introduction}

\noindent Given a family $\mr{L}$ of $r$-uniform hypergraphs ($r$-graphs for short), an $r$-graph is said to be {\it $\mr{L}$-free} if it contains no copy of any member in $\mr{L}$. The {\it extremal number} $\e_r(n,\mr{L})$ is the maximum number of edges in a $\mr{L}$-free $r$-graph on $n$ vertices. 
Katona, Nemetz, and Simonovits \cite{Katona-Nemetz-Simonovits-Turan-density} showed that for every $\mr{L}$, the {\it Tur\'an density}  $\pi_r(\mr{L}):=\lim_{n\rightarrow\infty}\frac{\e_r(n,\mr{L})}{\binom{n}{r}}$ always exists. If $\pi_r(\mr{L})>0$ then we call $\mr{L}$ {\it non-degenerate}; otherwise, we call it {\it degenerate}. For a degenerate $\mr{L}$, if there exists a real number $\alpha\in(0,r)$ and positive constants $c_1,c_2$ such that $c_1 n^\alpha\le \e_r(n,\mr{L})\le c_2 n^\alpha$, for all sufficiently large $n$, then we call $\alpha$ the {\it Tur\'an exponent} of $\mr{L}$. 
%Such an exponent, if it exists, is necessarily unique. 
Moreover, if the limit $$\pi_r^d(\mr{L}):=\lim_{n\rightarrow\infty}\frac{\e_r(n,\mr{L})}{n^{\alpha}}$$ exists and lies in $(0,\infty)$, then we call $\pi_r^d(\mr{L})$ the {\it degenerate Tur\'an density} of $\mr{L}$-free $r$-graphs. %The requirement that the limit is positive and finite makes the exponent $\alpha$ unique whenever it exists. 

In general, even if we know the Tur\'an exponent of a degenerate $\mr{L}$, it is still a difficult task to determine whether the corresponding degenerate density exists, not to mention its precise value. For example, let $C_i$ denote the cycle of length $i$. Classic results in extremal graph theory showed that $\e_2(n,C_3)=\lfloor n^2/4\rfloor$ \cite{Mantel-K_3} and $\e_2(n,C_4)=\frac{1}{2}n^{3/2}+O(n)$ \cite{reiman1958,erdos-C_4-1966,brown-C_4-1966,KST-1954}.
However, the asymptotics of the closely related $\e_2(n,\{C_3,C_4\})$ is yet unknown. A famous conjecture of Erd\H{o}s \cite{erdos1975some} in 1975 asserts that 
\begin{equation}\label{eq:erdos-conjecture}
    \lim_{n\rightarrow\infty}\frac{\e_2(n,\{C_3,C_4\})}{n^{3/2}}=\frac{1}{2\sqrt{2}}.
\end{equation}

\noindent Despite much effort in the last 50 years, not much progress has been made. It is not known whether the aforementioned limit exists. The current record is $$\frac{1}{2\sqrt{2}}n^{3/2}+\Omega(n^{1.25})\le\e_2(n,\{C_3,C_4\})\le\e_2(n,C_4)=\frac{1}{2}n^{3/2}+O(n),$$ where the lower bound is due to a recent work by Ma and Yang \cite{Ma-Yang25}.  

The {\it union-free hypergraphs} studied in this paper are an extension of $\{C_3,C_4\}$-free graphs to hypergraphs. 
For fixed positive integers $r,t$, an $r$-graph $\mathcal{H}$ is called {\it $t$-union-free} if for arbitrary two distinct subsets $\ma{A},\ma{B}\subseteq\mathcal{H}$, each consisting of at most $t$ edges, we have $\cup_{A\in\ma{A}} A\neq \cup_{B\in\ma{B}} B$. Let $U_t(n,r)$ denote the maximum number of edges in an $n$-vertex $t$-union-free $r$-graph. It is clear that a graph is $\{C_3,C_4\}$-free if and only if it is 2-union-free, and hence $U_2(n,2)=\e_2(n,\{C_3,C_4\})$. Our main concern is to study the degenerate Tur\'an density of general $t$-union-free $r$-graphs for $t\ge 3$ and $r\ge 3$.

The study of union-free hypergraphs has multiple origins, motivated by problems and applications in coding theory \cite{kautz-singleton-superimposed-codes}, combinatorics \cite{erdos1977problems}, and group testing \cite{hwang1987non}. Historically, in 1964, Kautz and Singleton \cite{kautz-singleton-superimposed-codes} first introduced the notion of $t$-union-free hypergraphs\footnote{with the name ``uniquely decipherable codes of order $t$''}, motivated by applications in information retrieval and data communication. Research on related topics is still an active area in coding theory, see \cite{Vorobyev-Ilya21,Daniil-Nikita-Ilya24,yachkov-V-P-S14}. Unaware of their work, in 1977, Bollob\'as and Erd\H{o}s (see page 11 of \cite{erdos1977problems}) posed the combinatorial problem of determining the maximum number of edges in an $n$-vertex $2$-union-free $r$-graph and showed that $U_2(n,3)=\Theta(n^2)$. In 1987, Hwang and S\'os \cite{hwang1987non} introduced the closely related notion {\it $t$-Sidon family}. A hypergraph has this property if any two distinct subsets of {\it exactly} $t$ edges have distinct union. 

For 2-union-free hypergraphs, Frankl and F\"uredi determined the limits $\lim_{n\to\infty}\frac{U_2(n,3)}{n^2}=\frac{1}{6}$ \cite{FRANKL-F84} and $\lim_{n\to\infty}\frac{U_2(n,4)}{n^3}=\frac{1}{24}$ \cite{Frankl-Furedi-2-union-free}. To the best of our knowledge, these are the only two cases for which the sharp asymptotic bounds on $U_t(n,r)$ are known. In \cite{Frankl-Furedi-2-union-free}, Frankl and F\"uredi also showed that    
\begin{equation}\label{eq:Frankl-Furedi-U_2(n,r)}
    U_2(n,r)=\Theta(n^{\lc 4r/3\rc/2}).
\end{equation}

For $t\ge 3$, $t$-union-free hypergraphs are closely related to an extensively studied combinatorial object, known as cover-free hypergraphs, introduced independently by Kautz and Singleton \cite{kautz-singleton-superimposed-codes} and Erd\H{o}s, Frankl, and F\"uredi \cite{Erdos-Frankl-Furedi-2-cover-free,Erdos-Frankl-Furedi-r-cover-free}. An $r$-graph $\mathcal{H}$ is {\it $t$-cover-free} if no edge is contained in the union of any $t$ other distinct edges. Cover-free and union-free hypergraphs have various applications in coding theory, group testing, and cryptography. The reader is referred to the survey \cite{idalino2022survey} for more details. 

The following fact connects union-free and cover-free hypergraphs \cite{Furedi-Ruszinko-no-grid}. 

\vspace{10pt}

\noindent\textbf{Fact.}  If $\ma{H}$ is $t$-cover-free then it is $t$-union-free; if $\ma{H}$ is $t$-union-free then it is $(t-1)$-cover-free.

\vspace{10pt}

Let $F_t(n,r)$ denote the maximum number of edges in an $n$-vertex $t$-cover-free $r$-graph. By the above fact, one can infer that
\begin{align}\label{eq:union-free-and-cover-free}
    F_t(n,r)\le U_t(n,r)\le F_{t-1}(n,r).
\end{align}

\noindent Frankl and F\"uredi \cite{Frankl-Furedi-colored-packing} proved sharp asymptotic bounds on $F_t(n,r)$ for all $t\ge 2,~r\ge 3$, i.e.,
\begin{equation}\label{eq:cff-frankl-furedi}
    \lim_{n\to \infty}\frac{F_{t}(n,r)}{n^{\lc\fr{r}{t}\rc}}=\frac{1}{\lc\fr{r}{t}\rc!} \cdot \frac{1}{\binom{r}{\lc\fr{r}{t}\rc}-m(r,\lc\fr{r}{t}\rc,\lambda)},
\end{equation}

\noindent where $0\le \lambda\le t-1$ is the unique integer satisfying $r=\lc\fr{r}{t}\rc(\lambda+1)+(\lc\fr{r}{t}\rc-1)(t-\lambda-1)$, and for positive integers $n,r,\lambda$, $m(n,r,\lambda)$ is the maximum number of edges in an $n$-vertex $r$-graph with no $\lambda+1$ pairwise disjoint edges (i.e., it has matching number at most $\lambda$).

It follows from \eqref{eq:union-free-and-cover-free} and \eqref{eq:cff-frankl-furedi} that
\begin{equation}\label{eq:cff-bound}
    \Omega(n^{\lceil\frac{r}{t}\rceil})=F_t(n,r)\le U_t(n,r)\le F_{t-1}(n,r)=O(n^{\lceil\frac{r}{t-1}\rceil}).
\end{equation}

\noindent There have been several improvements upon \eqref{eq:cff-bound} for different parameter regimes. In 2013, F\"uredi and Ruszink\'{o} \cite{Furedi-Ruszinko-no-grid} established nearly quadratic growth for the case $t=r$, i.e., they showed $n^{2-o(1)}<U_r(n,r)=O(n^2)$ for $r\ge 4$ and $U_3(n,3)=\Omega(n^{5/3})$. They also conjectured that $U_r(n,r)=o(n^2)$ (see \cite[Conjecture 1.11]{Furedi-Ruszinko-no-grid}). In 2020, Shangguan and Tamo \cite{Shangguan-Tamo-canc-and-union-free} proved that for all $t\ge 3$ and $r\ge 3$,
\begin{equation}\label{eq:Shangguan-Tamo-union-free}
    U_t(n,r)=\Omega(n^{\frac{r}{t-1}}).
\end{equation}

\noindent Putting \eqref{eq:union-free-and-cover-free}, \eqref{eq:cff-frankl-furedi}, \eqref{eq:cff-bound}, and \eqref{eq:Shangguan-Tamo-union-free} together, it is not hard to see that for $t\ge 2$, $k\ge 2$, and $0\le a\le t-1$, 
\begin{equation}\label{eq:ST-union-free}
    \begin{aligned}
        \Omega(n^{\max\left\{\lceil\frac{tk-a}{t+1}\rceil,\frac{tk-a}{t}\right\}})&=U_{t+1}(n,tk-a)\\&\le F_t(n,tk-a)\le\frac{1}{k!}\cdot \frac{1}{\binom{tk-a}{k}-m(tk-a,k,t-a-1)}\cdot n^k+o(n^k).
    \end{aligned}
\end{equation}
\noindent Note that for notational convenience, in \eqref{eq:ST-union-free} and below we use $U_{t+1}(n,\cdot)$ instead of $U_t(n,\cdot)$.

Our main contribution is a new lower bound for $U_{t+1}(n,r)$, showing that for almost all $t\ge 2,~r\ge 3$, except for $\{U_3(n,2s-1),U_{s+1}(n,s+1),U_{s+1}(n,s+2):s\ge 2\}$ and $U_4(n,7)$,
the upper bound in \eqref{eq:ST-union-free} is asymptotically sharp, up to a lower order term. In other words, we determine the degenerate Tur\'an densities of these union-free hypergraphs.

\begin{theorem}\label{thm:main-union-free}
    Let $t\ge 2$, $k\ge 2$, and $0\le a \le t-1$ be integers. Then we have
    \begin{align}\label{eq:union-free}
        \lim_{n\to \infty}\frac{U_{t+1}(n,tk-a)}{n^k}=\frac{1}{k!}\cdot \frac{1}{\binom{tk-a}{k}-m(tk-a,k,t-a-1)},
    \end{align}
    except for $(t+1,tk-a)\in\{(3,2s-1),(s+1,s+1),(s+1,s+2):s\ge 2\}\cup\{(4,7)\}.$
\end{theorem}

Below we briefly discuss the exact value of $m(r,k,\lambda)$. In 1965, Erd\H{o}s \cite{erdos1965problem} proposed the following well-known conjecture.
\begin{conjecture}[Erd\H{o}s Matching Conjecture \cite{erdos1965problem}]
    For all positive integers $r,k,\lambda$ with $r\ge (\lambda+1)k$,
    $$m(r,k,\lambda)=\max\left\{\binom{r}{k}-\binom{r-\lambda}{k},\binom{(\lambda+1)k-1}{k}\right\}.$$
\end{conjecture}

\noindent This conjecture has been extensively studied and is now partially solved, see \cite{Frankl-matching,frankl2013improved,frankl2017maximum,Frankl-Kupavskii-The-Erdos-Matching-Con,kleitman1968maximal,EMC2023Dmitriy-Andrey,EMC2025Ryan-Bal,alon-frankl-huang-rodl-2012large,frankl2026towards,hou2026finite} and the references therein. When $\lambda=0$, we naively have $m(r,k,0)=0$. When $\lambda=1$, the celebrated Erd\H{o}s-Ko-Rado theorem \cite{erdos-ko-rado1961} implies that $m(r,k,1)=\binom{r-1}{k-1}$. When $k=2,3$, the conjecture was resolved by Erd\H{o}s and Gallai \cite{erdos-gallai1959} and by Frankl \cite{frankl2017maximum}, respectively. Recently, Frankl, Lu, Ma, and Wu \cite{frankl2026towards} and Hou, Hu, and Liu \cite{hou2026finite} have made substantial progress for $k=4$. For general $k$, the current records can be found in \cite{kleitman1968maximal,frankl2013improved,Frankl-Kupavskii-The-Erdos-Matching-Con,Frankl-matching}.

\iffalse
Very recently, Frankl, Lu, Ma, and Wu proved the conjecture for $k=4$ and sufficiently large $n\ge 5s$ \cite{frankl2026towards}, and Hou, Hu, and Liu \cite{hou2026finite} proved the conjecture for $s\ge 6961$. 
For general $k$, the following results are known:

\begin{itemize}
    \item when $r=tk$, $m(tk,k,t-1)=\binom{tk-1}{k}$ \cite{kleitman1968maximal};
    \item when $r\ge (2\lambda +1)k -\lambda$, $m(r,k,\lambda)=\binom{r}{k}-\binom{r-\lambda}{k}$ \cite{frankl2013improved};
    \item for sufficiently large $\lambda$ and $r\ge \frac{5}{3}\lambda k -\frac{2}{3}\lambda$, $m(r,k,\lambda)=\binom{r}{k}-\binom{r-\lambda}{k}$ \cite{Frankl-Kupavskii-The-Erdos-Matching-Con};
    \item for $(\lambda +1)k\le r\le (\lambda +1)(k+\epsilon)$, where $\epsilon$ depends on $k$, $m(r,k,\lambda)=\binom{(\lambda +1)k-1}{k}$ \cite{Frankl-matching}.
\end{itemize}
\fi

Substituting the exact value of $m(tk-a,k,t-a-1)$ into Theorem~\ref{thm:main-union-free} gives us the exact value of the corresponding $\lim_{n\to \infty}\frac{U_{t+1}(n,tk-a)}{n^k}$. We list some of them as follows.
\begin{corollary}
    %Let $t\ge 2$, $k\ge 2$, and $0\le a\le t-1$ be integers. 
    Under the assumptions of Theorem~\ref{thm:main-union-free}, and excluding the exceptional cases listed in that theorem, then the following explicit values of $\lim_{n\to\infty}\frac{U_{t+1}(n,tk-a)}{n^k}$ hold:%, whenever the corresponding parameters are not among the exceptional cases in Theorem~\ref{thm:main-union-free}:
    \begin{itemize}
        \item For $k=2$, Erd\H{o}s and Gallai \cite{erdos-gallai1959} proved that
        $$m(2t-a,2,t-a-1)=\max\left\{\binom{2t-2a-1}{2},\binom{2t-a}{2}-\binom{t+1}{2}\right\}.$$
        Therefore, by Theorem~\ref{thm:main-union-free}, we have
        $$\lim_{n\to\infty}\frac{U_{t+1}(n,2t-a)}{n^2}=
        \begin{cases}
        \displaystyle
        \frac{1}{2}\cdot\frac{1}{\binom{t+1}{2}}, %=\frac{1}{t(t+1)},
        & \text{if } t\le 3a+2,\\[1.5ex]
        \displaystyle
        \frac{1}{2}\cdot\frac{1}{\binom{2t-a}{2}-\binom{2t-2a-1}{2}},
        %=\frac{1}{(a+1)(4t-3a-2)}, 
        & \text{if } t>3a+2.
        \end{cases}
        $$
        \item For $k=3$, Frankl \cite{frankl2017maximum} proved that 
        $$m(3t-a,3,t-a-1)=\max\left\{\binom{3t-3a-1}{3},\binom{3t-a}{3}-\binom{2t+1}{3}\right\}.$$
        Therefore, by Theorem~\ref{thm:main-union-free},
        $$\lim_{n\to\infty}\frac{U_{t+1}(n,3t-a)}{n^3}=
        \begin{cases}
        \displaystyle
        \frac{1}{6}\cdot\frac{1}{\binom{2t+1}{3}}, & \text{if } -26a^2+46at-25a-8t^2+19t-6\ge0,\\[1.5ex]
        \displaystyle
        \frac{1}{6}\cdot\frac{1}{\binom{3t-a}{3}-\binom{3t-3a-1}{3}}, & \text{if } -26a^2+46at-25a-8t^2+19t-6<0.
        \end{cases}
        $$
        \item For $a=0$, Kleitman \cite{kleitman1968maximal} proved $m(tk,k,t-1)=\binom{tk-1}{k}$. Therefore,
        $$\lim_{n\to\infty}\frac{U_{t+1}(n,tk)}{n^k}=
        %\frac{1}{k!}\cdot\frac{1}{\binom{tk}{k}-\binom{tk-1}{k}}=
        \frac{1}{k!}\cdot\frac{1}{\binom{tk-1}{k-1}}.$$
        \item When $tk-a\ge (2(t-a-1)+1)k-(t-a-1)$, equivalently, $\frac{t-1}{2}\le a \le t-1$, Frankl \cite{frankl2013improved} proved $m(tk-a,k,t-a-1)=\binom{tk-a}{k}-\binom{tk-a-(t-a-1)}{k}$. Therefore,
        $$\lim_{n\to\infty}\frac{U_{t+1}(n,tk-a)}{n^k}=
        %\frac{1}{k!}\cdot\frac{1}{\binom{tk-a}{k}-\left(\binom{tk-a}{k}-\binom{tk-a-(t-a-1)}{k}\right)}=
        \frac{1}{k!}\cdot\frac{1}{\binom{t(k-1)+1}{k}}.$$
    \end{itemize}
\end{corollary}

\paragraph{Idea of the proof.} The proof of the lower bound in Theorem~\ref{thm:main-union-free} is based on the method of hypergraph packings (see Section~\ref{sec:definitions-lemmas} below for a formal definition), which is a powerful technique for establishing probabilistic lower bounds. To the best of our knowledge, this line of research began with Frankl and R\"odl \cite{Frankl-Rodl-matching}, who proved the existence of {\it near-optimal} $\mathcal{J}$-packings for any fixed hypergraph $\mathcal{J}$, generalizing R\"odl's work on the existence of approximate Steiner systems \cite{Rodl-nibble}. It was later extended by Frankl and F\"uredi \cite{Frankl-Furedi-colored-packing}, who obtained {\it near-optimal induced} $\mathcal{J}$-packings and used them to establish a sharp asymptotic lower bound for cover-free hypergraphs (see \eqref{eq:cff-frankl-furedi}). Based on a powerful result proved independently by Delcourt and Postle \cite{Delcourtconflict} and by Glock, Joos, Kim, K\"uhn, and Lichev \cite{Glockconflict} (see Lemma \ref{lem:useful-lem} below), we strengthen \cite{Frankl-Furedi-colored-packing} by proving the existence of {\it near-optimal locally sparse induced} $\mathcal{J}$-packings (see Lemma~\ref{lem:main-lemma}) and applying it to derive a sharp asymptotic lower bound for union-free hypergraphs. Similar locally sparse packing results have recently been used by several authors to attack an extremal hypergraph problem of Brown-Erd\H{o}s-S\'os (see \cite{bes,Glock64,Glock-Kim-Lichev-Pikhurko-Sun-bes,Letzter-Sgueglia-23,pikhurko-sun} for more details).

\paragraph{Organization.} The paper is organized as follows. In Section~\ref{sec:definitions-lemmas}, we present the necessary definitions and a brief introduction to the history of hypergraph packings and matchings. We also state a useful lemma (Lemma~\ref{lem:useful-lem}) from \cite{Delcourtconflict,Glockconflict} on {\it conflict-free} hypergraph matchings. In Section~\ref{sec:main-lemma}, we establish our main technical lemma (Lemma~\ref{lem:main-lemma}) on locally sparse induced hypergraph packings, using Lemma~\ref{lem:useful-lem}. The proof of the lower bound in Theorem~\ref{thm:main-union-free} is divided into the case $a=0$ (proved in Section~\ref{sec:a=0}) and the case $1\le a\le t-1$, whose technically simpler proof is deferred to  Appendix~\ref{sec:a>0}. Finally, we mention some open problems in Section~\ref{sec:conclusion}.

\section{Preliminaries}\label{sec:definitions-lemmas}

\noindent Throughout, we assume that $r,t,k,v,e$ are fixed positive integers and $n\rightarrow\infty$. We use $[n]$ to denote the set $\{1,\ldots,n\}$. For a finite set $V$, we use $\binom{V}{r}$ to denote the family of all $r$-subsets of $V$. An $r$-graph $\mathcal{H}$ on the vertex set $V(\mathcal{H})$ is a family of distinct $r$-subsets of $V(\mathcal{H})$, that is, $\mathcal{H}\subseteq\binom{V(\mathcal{H})}{r}$. 

Denote the {\it symmetric difference} of two finite subsets $B,C$ by $B\Delta C:=(B\setminus C)\cup(C\setminus B)$. 

\paragraph{$(v,e)$-free hypergraphs.} An {\it $(v,e)$-configuration} is a family of $e$ distinct edges whose union contains at most $v$ vertices. We say an $r$-graph $\mathcal{H}$ is {\it $(v,e)$-free} if it contains no $(v,e)$-configuration, i.e., for arbitrary $e$ distinct edges $A_1,\ldots,A_e\in\mathcal{H}$, we have $|\cup_{i=1}^e A_i|\ge v+1$. A $(v,e)$-free hypergraph is also called {\it locally sparse}. 

The following lemma guarantees the existence of an $m$-vertex locally sparse $r$-graph with many edges. The large number of edges is needed to obtain the desired lower bound in our construction, while local sparsity rules out small configurations that could otherwise lead to a failure of union-freeness.

\begin{lemma}[see e.g. Proposition 6 in \cite{Shangguan-Tamo-spa-hyp-coding}]\label{lem:e^--free}
    For all fixed integers $r>k\ge 2$ and $e\ge 2$, there exists an $r$-graph $\mathcal{G}\subseteq\binom{[m]}{r}$ with at least $cm^{k+\frac{1}{e-1}}$ edges that is $(\ell r-(\ell-1)k-1,\ell)$-free for all $2\le\ell\le e$, where $c=c(r,k,e)$ is an absolute constant independent of $m$. 
\end{lemma}

%We briefly give the idea of the proof of Lemma~\ref{lem:e^--free}. Choose every $r$-subset of $[m]$ independently with probability $p=c_0m^{-r+k+\frac{1}{e-1}}$, where $c_0>0$ is a sufficiently small constant. The expected number of selected edges is $\left(\frac{c_0}{r!}+o(1)\right)m^{k+\frac{1}{e-1}}$. For each $2\le \ell\le e$, a straightforward counting argument shows that the expected number of collections of $\ell$ selected edges whose union has size at most $\ell r-(\ell-1)k-1$ is $O\left(m^{\ell r-(\ell-1)k-1}p^\ell\right)=O\left(c_0^\ell m^{k-1+\frac{\ell}{e-1}}\right)$. For $\ell<e$, this expectation is $o\left(m^{k+\frac{1}{e-1}}\right)$. For $\ell=e$, it is $O\left(c_0^e m^{k+\frac{1}{e-1}}\right)$. Thus, by choosing $c_0$ sufficiently small, the expected number of selected edges is larger than the expected total number of forbidden configurations by a positive constant multiple of $m^{k+\frac{1}{e-1}}$. Hence, there exists a realization from which one can delete one edge from every forbidden configuration while retaining at least $cm^{k+\frac{1}{e-1}}$ edges. The resulting hypergraph satisfies the conclusion of Lemma~\ref{lem:e^--free}.

\paragraph{Hypergraph packings and matchings.} Next, we present several definitions of hypergraph packings and matchings. 

\begin{definition}[Packings and induced packings]\label{def:packing}
    For a fixed $k$-graph $\mathcal{J}$, a {\it $\mathcal{J}$-packing} $\mathcal{P}$ in $\binom{[n]}{k}$ is a family of {\it edge-disjoint} copies of $\mathcal{J}$, that is, $\mathcal{P}=\{(V(\mathcal{J}_i),\mathcal{J}_i):i\in [|\mathcal{P}|]\}$, where for each $i$, $\mathcal{J}_i\subseteq\binom{[n]}{k}$ is a copy of $\mathcal{J}$, and for distinct $i,i'$, $\mathcal{J}_i\cap\mathcal{J}_{i'}=\emptyset$. 
    
    The $\mathcal{J}$-packing $\mathcal{P}$ is further called {\it induced} if for all distinct $i,i'$, $|V(\mathcal{J}_i)\cap V(\mathcal{J}_{i'})|\le k$, and if $|V(\mathcal{J}_i)\cap V(\mathcal{J}_{i'})|=k$ then $V(\mathcal{J}_i)\cap V(\mathcal{J}_{i'})$ is neither an edge in $\mathcal{J}_i$ nor an edge in $\mathcal{J}_{i'}$.
\end{definition}

If $\mathcal{J}$ has $m$ vertices and $e\ge2$, we say that a $\mathcal{J}$-packing $\mathcal{P}=\{(V(\mathcal{J}_i),\mathcal{J}_i):i\in[|\mathcal{P}|]\}$ in $\binom{[n]}{k}$ is {\it $e$-locally sparse} if the $m$-graph $\{V(\mathcal{J}_i):i\in[|\mathcal{P}|]\}\subseteq\binom{[n]}{m}$ is $(\ell m-(\ell-1)k-1,\ell)$-free for every $2\le \ell\le e$. 

Since the $k$-graphs in a $\mathcal{J}$-packing are pairwise edge-disjoint, any $\mathcal{J}$-packing in $\binom{[n]}{k}$ can have at most $\frac{\binom{n}{k}}{|\mathcal{J}|}$ copies of $\mathcal{J}$. The influential work of R\"odl \cite{Rodl-nibble}, Frankl and R\"odl \cite{Frankl-Rodl-matching}, and Pippenger (see \cite{Pippenger-Spencer-asymptotic}) showed that this trivial upper bound is essentially tight: there exists a near-optimal $\mathcal{J}$-packing containing at least $(1-o(1))\cdot\frac{\binom{n}{k}}{|\mathcal{J}|}$ edge-disjoint copies of $\mathcal{J}$, where $o(1)\rightarrow0$ as $n\rightarrow\infty$. Frankl and F\"uredi \cite{Frankl-Furedi-colored-packing} strengthened this result by proving the existence of near-optimal induced $\mathcal{J}$-packings. 

Very recently, Delcourt and Postle \cite{Delcourtconflict} and independently Glock, Joos, Kim, K\"uhn, and Lichev \cite{Glockconflict} greatly extended the above line of work by establishing a general framework that shows the existence of near-optimal $\mathcal{J}$-packings {\it avoiding certain small configurations}. Our proof of the lower bound on $U_t(n,r)$ also applies this framework. 

We need some preparations before stating the result of \cite{Delcourtconflict,Glockconflict}. For reals $a,b\in(0,1)$, write $a=1\pm b$ if $a\in[1-b,1+b]$. For an integer $r\ge 1$, a hypergraph is {\it $r$-bounded} if every edge
has size at most $r$. For a hypergraph $\mathcal{C}$ (not necessarily uniform) and an integer $\ell\ge1$, let $\mathcal{C}^{(\ell)}$ denote the subhypergraph of $\mathcal{C}$ formed by the $\ell$-edges in $\mathcal{C}$, i.e., $V(\mathcal{C}^{(\ell)})=V(\mathcal{C})$ and $\mathcal{C}^{(\ell)}=\{C\in\mathcal{C}:|C|=\ell\}$. For a subset $X\subseteq V(\mathcal{C})$, let $\deg_{\mathcal{C}}(X)$ denote the number of edges in $\mathcal{C}$ that contain $X$, i.e., $\deg_{\mathcal{C}}(X)=|\{C\in \mathcal{C}:X\subseteq C\}|$. If $X=\{v\}$ is a single vertex, then we write $\deg_{\mathcal{C}}(v)=\deg_{\mathcal{C}}(\{v\})$. Let $\Delta_{\ell}(\mathcal{C})$ be the maximum of $\deg_{\mathcal{C}}(X)$ over all $X\in\binom{V(\mathcal{C})}{\ell}$. 

Given a hypergraph $\mathcal{H}$, a {\it matching} in $\mathcal{H}$ is a set of pairwise disjoint edges of $\mathcal{H}$. Let $\mathcal{C}$ be a hypergraph defined on the {\it edges} of $\mathcal{H}$. A matching $\mathcal{M}\subseteq\mathcal{H}$ is said to be $\mathcal{C}$-free if $\mathcal{M}$ does not contain $\mathcal{C}$ as a subhypergraph. Such a matching is called a {\it conflict-free} hypergraph matching (\cite{Delcourtconflict,Glockconflict}).

The lemma below gives sufficient conditions for which a hypergraph $\mathcal{H}$ has a near-optimal $\mathcal{C}$-free matching. In our application, the matching gives the desired packing, while the forbidden conflicts $\mathcal{C}$ encode the local sparsity conditions. 

\begin{lemma}[see Corollary 1.17 in \cite{Delcourtconflict} and Theorem 1.3 in \cite{Glockconflict}]\label{lem:useful-lem}
    For all integers $r,e\ge 2$ and every real $\ep>0$, there exist an integer
    $D_0\ge 0$ and a real $\alpha>0$ such that the following holds for all
    $D\ge D_0$. Let $\mathcal{H}$ be an $r$-bounded multi-hypergraph satisfying $\Delta(\mathcal{H})\le D$ and $\Delta_2(\mathcal{H})\le D^{1-\ep}$. 

    Let $\mathcal{C}$ be a hypergraph with $V(\mathcal{C})=\mathcal{H}$ such that every $C\in\mathcal{C}$ satisfies $2\le|C|\le e$, and the following conditions hold:
    \begin{itemize}
        \item [(i)] $\Delta_1(\mathcal{C}^{(\ell)})\le \alpha D^{\ell -1}\log D$ for all $2\le \ell\le e$;
        \item [(ii)] $\Delta_{\ell'}(\mathcal{C}^{(\ell)})\le D^{\ell-\ell'-\epsilon}$ for all $2\le \ell'<\ell\le e$;
        \item [(iii)] $|\{H_1\in\mathcal{H}:\{H_1,H_2\}\in\mathcal{C} \mbox{ and } v\in H_1\}|\le D^{1-\ep}$ for all $H_2\in \mathcal{H}$ and $v\in V(\mathcal{H})$;
        \item [(iv)] $|\{H_1\in\mathcal{H}:\{H_1,H_2\},\{H_1,H_3\}\in\mathcal{C}\}|\le D^{1-\ep}$ for all disjoint $H_2,H_3\in \mathcal{H}$.
    \end{itemize}
    Then, there exists a $\mathcal{C}$-free matching of $\mathcal{H}$ with size at least $\frac{|\mathcal{H}|}{D}\bigl(1-D^{-\alpha}\bigr)$.
\end{lemma}

%To establish the lower bounds on $U_t(n,r)$, our argument relies on an induced hypergraph packing avoiding small $(v,e)$-configurations. In the next section, we prove the existence of such a packing via Lemma~\ref{lem:useful-lem}.

\section{Locally sparse induced hypergraph packings}\label{sec:main-lemma}

\noindent In this section, we present our main technical lemma used in the proofs of Theorem \ref{thm:main-union-free}. It shows that for any fixed $k$-graph $\mathcal{J}$ with at least two edges, there exists a near-optimal induced $\mathcal{J}$-packing that avoids small $(v,e)$-configurations. 

\begin{lemma}\label{lem:main-lemma}
    Let $m> k,~e\ge 2$ be fixed integers and $\mathcal{J}$ be a fixed $k$-graph on $m$ vertices with at least $2$ edges. Then for every $\epsilon>0$ there exists an integer $n_0$ such that the following holds for all $n\ge n_0$. There exists a $\mathcal{J}$-packing $\mathcal{P}=\{(V(\mathcal{J}_i),\mathcal{J}_i):i\in[|\mathcal{P}|]\}$ in $\binom{[n]}{k}$ such that the following conditions hold:  
    \begin{itemize}
        \item [(i)] $\mathcal{P}$ is induced; 
        \item [(ii)] $\mathcal{P}$ is $e$-locally sparse; 
        \item [(iii)] $\mathcal{P}$ is near-optimal, i.e., $|\mathcal{P}|\ge(1-\epsilon)\binom{n}{k}/|\mathcal{J}|$.
        %, where $o(1)\rightarrow 0$ as $n\rightarrow\infty$.
    \end{itemize}
\end{lemma}

Lemma~\ref{lem:main-lemma} strengthens several classical packing results by adding the local sparsity condition. Specifically, the works of R\"odl \cite{Rodl-nibble}, Frankl and R\"odl \cite{Frankl-Rodl-matching}, and Pippenger \cite{Pippenger-Spencer-asymptotic} establish the existence of near-optimal packings, as required in condition (iii). Frankl and F\"uredi \cite{Frankl-Furedi-colored-packing} further showed that such a packing can be chosen to be induced, thus giving conditions (i) and (iii). Lemma~\ref{lem:main-lemma} adds the local sparsity condition (ii), and hence contains these results as special cases.

Below, we present the proof of Lemma~\ref{lem:main-lemma}, which applies Lemma~\ref{lem:useful-lem}. 

\subsection{Proof of Lemma~\ref{lem:main-lemma}}
     \noindent Throughout the proof, we assume that $n$ is sufficiently large. Choose a constant $\gamma\in(0,\min\{\frac{\ep}{8},\frac{1}{2(m-k)}\})$.
     %$$0<\gamma<\min\{\frac{\ep}{8},\frac{1}{2(m-k)}\}.$$ 
     Color every edge of the complete $k$-graph $\binom{[n]}{k}$ red with probability $\gamma$ and blue with probability $1-\gamma$. 
     All colorings are independent. Let $\mathcal{H}$ be the labelled $|\mathcal{J}|$-uniform multi-hypergraph whose vertices are the blue edges in $\binom{[n]}{k}$. The edges of $\mathcal{H}$ are the labelled copies $\mathcal{J}'$ of $\mathcal{J}$ with $V(\mathcal{J}')\subseteq[n]$ such that every edge of $\mathcal{J}'$ is colored blue and every non-edge of $\binom{V(\mathcal{J}')}{k}\setminus\mathcal{J}'$ is colored red. Here and throughout the proof, different labelled copies are regarded as distinct parallel edges of $\mathcal H$, even if they have the same set of vertices in $\mathcal H$. All degrees in $\mathcal H$ are counted with multiplicity. %Multiple edges are allowed and counted as distinct labelled occurrences. %Equivalently, one may replace each labelled occurrence by a distinct formal edge; all degrees below count these labelled occurrences. Thus isolated vertices of $\mathcal{J}$ cause no loss in the degree counts. In Claim~\ref{claim:matching-packing} below, we show that this coloring ensures that matchings in $\mathcal{H}$ correspond to induced $\mathcal{J}$-packings as in Lemma~\ref{lem:main-lemma} (i).

     To establish Lemma~\ref{lem:main-lemma} (ii), we define a conflict hypergraph $\mathcal{C}$ whose vertices are the labelled edges of $\mathcal{H}$. The edges of $\mathcal{C}$, which we call conflicts, are the sets of labelled copies of $\mathcal{J}$ in $\mathcal{H}$ whose vertex sets form a {\it minimal} bad configuration; that is, their vertex sets form an $(\ell m-(\ell-1)k-1,\ell)$-configuration for some $2\le \ell\le e$ but contain no $(\ell' m-(\ell'-1)k-1,\ell')$-configuration for any $1\le \ell'< \ell$. 
     %The threshold $\ell m-(\ell-1)k-1$ is chosen so that a family of selected copies fails to be $e$-locally sparse precisely when some $\ell$ of their vertex sets form such a bad configuration. 
     More precisely,  
     \begin{equation}\label{eq:def-C}
         \begin{aligned}
            \mathcal{C}:=\bigcup_{\ell=2}^{e}\Bigg\{\{\mathcal{J}_1,\ldots,\mathcal{J}_{\ell}\} \in\binom{\mathcal{H}}{\ell}:
            &\left|\bigcup_{i=1}^{\ell}V(\mathcal{J}_i)\right|\le\ell m-(\ell-1)k-1,\\
            &\left|\bigcup_{i\in S}V(\mathcal{J}_i)\right|\ge|S|m-(|S|-1)k \mbox{ for all }\emptyset \not=S\subsetneq [\ell]\Bigg\}.
        \end{aligned}
     \end{equation}

    The following claim connects matchings in $\mathcal{H}$ to $\mathcal{J}$-packings in $\binom{[n]}{k}$.
 
    \begin{claim}\label{claim:matching-packing}
        Every $\mathcal{C}$-free matching in $\mathcal{H}$ yields a $\mathcal{J}$-packing in $\binom{[n]}{k}$ that satisfies Lemma~\ref{lem:main-lemma} (i), (ii). 
    \end{claim}

    \begin{proof}
    Let $\mathcal{M}$ be a $\mathcal{C}$-free matching in $\mathcal{H}$. The edges in $\mathcal{M}$ are labelled copies of $\mathcal{J}$. Notice that parallel edges of $\mathcal H$ cannot both belong to $\mathcal M$, since they have the same vertex set in $\mathcal H$. These copies are pairwise edge disjoint as $k$-graphs in $\binom{[n]}{k}$, since they are pairwise vertex disjoint as edges in $\mathcal{H}$. Therefore, the edges in $\mathcal{M}$ yield a $\mathcal{J}$-packing in $\binom{[n]}{k}$. Denote this packing by $\mathcal{P}$. %Since $\mathcal{M}$ is $\mathcal{C}$-free, $\mathcal{P}$ satisfies Lemma~\ref{lem:main-lemma} (ii). 

    We first show that $\mathcal{P}$ is $e$-locally sparse. Assume for contradiction that $\mathcal P$ is not $e$-locally sparse. Then there exists a subfamily $\mathcal{Q}\subseteq \mathcal{P}$ with  $2\le |\mathcal{Q}|\le e$ such that $|\cup_{\mathcal{J}'\in \mathcal{Q}}V(\mathcal{J}')|\le |\mathcal{Q}| m-(|\mathcal{Q}|-1)k-1$. Choose such a subfamily $\mathcal{Q}$ with minimum cardinality. Then every nonempty proper subfamily $\mathcal{Q}'\subsetneq\mathcal{Q}$ satisfies $|\cup_{\mathcal{J}'\in \mathcal{Q}'}V(\mathcal{J}')|\ge |\mathcal{Q'}| m-(|\mathcal{Q'}|-1)k$. Thus this minimal subfamily $\mathcal{Q}\subseteq \mathcal{P}$ is an edge of $\mathcal{C}$ contained in $\mathcal{M}$, contradicting the fact that $\mathcal{M}$ is $\mathcal{C}$-free. Therefore $\mathcal{P}$ is $e$-locally sparse. 
    
    It remains to show that $\mathcal{P}$ is induced. Indeed, let $\mathcal{J}'$ and $\mathcal{J}''$ be two distinct copies of $\mathcal{J}$ in $\mathcal{P}$. As $\mathcal{P}$ satisfies Lemma~\ref{lem:main-lemma} (ii) with $\ell=2$, we have $|V(\mathcal{J}')\cap V(\mathcal{J}'')|\le k$. Next, we show that if $|V(\mathcal{J}')\cap V(\mathcal{J}'')|=k$, then $V(\mathcal{J}')\cap V(\mathcal{J}'')$ is neither an edge in $\mathcal{J}'$ nor an edge in $\mathcal{J}''$. For the sake of contradiction, assume that $V(\mathcal{J}')\cap V(\mathcal{J}'')$ is an edge in $\mathcal{J}'$. Then, by definition, it is colored blue in $\mathcal{J}'$. However, since $\mathcal{J}'$ and $\mathcal{J}''$ are edge disjoint copies of $\mathcal{J}$, $V(\mathcal{J}')\cap V(\mathcal{J}'')$ is a non-edge of $\mathcal{J}''$, and hence has to be colored red in $\mathcal{J}''$, which is a contradiction. By symmetry, $V(\mathcal{J}')\cap V(\mathcal{J}'')$ is not an edge of $\mathcal J''$ either. Therefore, $\mathcal P$ is induced.
    \end{proof}
     
    Given Claim~\ref{claim:matching-packing}, to prove Lemma~\ref{lem:main-lemma}, it suffices to show the existence of a near-optimal $\mathcal{C}$-free matching in $\mathcal{H}$. From the definition of $\mathcal{H}$, $|V(\mathcal{H})|=\sum_{K\in \binom{[n]}{k}}X_{K}$ is a sum of $\binom{n}{k}$ independent indicator random variables, where $X_{K}=1$ if $K$ is colored blue and $X_{K}=0$ if $K$ is colored red. It is clear that $\mathbb{E}[|V(\mathcal{H})|]=(1-\gamma)\binom{n}{k}$, and by a Chernoff bound, with high probability we have $|V(\mathcal{H})|=(1\pm 2\gamma)\binom{n}{k}$. 

    \begin{claim}\label{claim-concentration}
         With high probability, every vertex in $\mathcal{H}$ is contained in $(1\pm d^{-\gamma})d$ edges, where $d=(c+o(1))n^{m-k}$ and $c$ is a constant depending on $\mathcal{J},\gamma,m,k$. 
    \end{claim}
    \noindent The proof of Claim~\ref{claim-concentration} is postponed to Appendix~\ref{sec:claim}. 

    Write $D:=(1+d^{-\gamma})d=(c+o(1))n^{m-k}$. We apply Lemma~\ref{lem:useful-lem} with $r=|\mathcal J|$ and $\epsilon=\gamma$. %Let $D_0$ and $\alpha>0$ be the constants given by Lemma~\ref{lem:useful-lem}. Since $D\to\infty$, we have $D\ge D_0$ for all sufficiently large $n$. 
    Claim~\ref{claim-concentration} gives $\Delta(\mathcal H)\le D$ with high probability.
    We next bound $\Delta_2(\mathcal H)$. Let $T_1,T_2\in V(\mathcal H)$ be distinct. Then $|T_1\cup T_2|\ge k+1$. An edge of $\mathcal H$ containing both $T_1$ and $T_2$ is a labelled copy of $\mathcal J$ whose ground vertex set contains $T_1\cup T_2$. For each fixed ground vertex set, there are only constantly many such labelled copies, where the constant depends only on $\mathcal J$. This count includes parallel edges. Hence, for some constant $C_0=C_0(\mathcal J)$,
    %We verify the hypotheses of Lemma~\ref{lem:useful-lem} with degree parameter $D$, $\alpha=1$, and  exponent $\gamma$. The hypergraph $\mathcal{H}$ is an $|\mathcal{J}|$-uniform multi-hypergraph, and Claim~\ref{claim-concentration} gives $\Delta(\mathcal{H})\le D$ with high probability. If $T_1,T_2\in V(\mathcal{H})$ are distinct, then $|T_1\cup T_2|\ge k+1$. Once $T_1$ and $T_2$ are fixed, an edge of $\mathcal{H}$ containing both of them uses at most $m-k-1$ further vertices, and there are only constantly many relevant labelled occurrences on a fixed vertex set. Consequently, for some constant $C_0=C_0(\mathcal{J})$, 
    $$\deg _{\mathcal{H}}(\{T_1,T_2\})=|\{\mathcal{J}'\in E(\mathcal{H}):\{T_1,T_2\}\subseteq\mathcal{J}'\}|\le C_0n^{m-k-1}\le D^{1-\gamma},$$
    \noindent where the last inequality holds for all sufficiently large $n$ because $\gamma<\frac{1}{2(m-k)}$. Thus, $\Delta_2(\mathcal H)\le D^{1-\gamma}$.
 
    It remains to verify that $\mathcal{H}$ and $\mathcal{C}$ satisfy Lemma~\ref{lem:useful-lem} (i)-(iv). Fix an edge $\{\mathcal{J}_1,\ldots,\mathcal{J}_{\ell}\}\in \mathcal{C}$ and a set $S\subsetneq [\ell]$ with $1\le |S|=\ell'<\ell$. Then $\cup_{i\in [\ell]\setminus S}V(\mathcal{J}_i)$ contains at most $(\ell-\ell')(m-k)-1$ vertices outside $\cup _{i\in S}V(\mathcal{J}_i)$. Indeed, 
    \begin{equation}\label{eq:obs}
        \begin{aligned}
             \left|\left(\bigcup_{i\in [\ell]\setminus S}V(\mathcal{J}_i)\right)\setminus\left(\bigcup_{j\in S}V(\mathcal{J}_j)\right)\right|&= \left|\left(\bigcup_{i\in [\ell]}V(\mathcal{J}_i)\right)\setminus\left(\bigcup_{j\in S}V(\mathcal{J}_j)\right)\right|\\
         &\le (\ell m-(\ell-1)k-1)-(\ell' m-(\ell'-1)k)\\
         &=(\ell-\ell')(m-k)-1,
    \end{aligned}
    \end{equation}
    
    \noindent where the inequality follows from the definition of $\mathcal{C}$ (see \eqref{eq:def-C}).
    Therefore, there is a constant $C_1=C_1(\mathcal{J},e)$ such that, for fixed $1\le \ell'< \ell$ and every subset $\mathcal{S}\subseteq V(\mathcal{C})=\mathcal{H}$ with size $\ell'$, we have 
     \begin{equation}\label{eq:c1-c4}
         \begin{aligned}
         \deg_{\mathcal{C}^{(\ell)}}(\mathcal{S})=|\{\{\mathcal{J}_1,\ldots,\mathcal{J}_{\ell}\}\in\mathcal{C}^{(\ell)}:\mathcal{S}\subseteq \{\mathcal{J}_1,\ldots,\mathcal{J}_{\ell}\}\}|\le  C_1n^{(\ell-\ell')(m-k)-1},
     \end{aligned}
     \end{equation}
     
     \noindent where the inequality follows from \eqref{eq:obs}. Indeed, given $\mathcal{S}$, there are at most $n^{(\ell-\ell')(m-k)-1}$ choices for the set $(\cup_{\mathcal{J}_i\in \{\mathcal{J}_1,\ldots,\mathcal{J}_{\ell}\}\setminus\mathcal{S}}V(\mathcal{J}_i))\setminus(\cup_{\mathcal{J}_i\in \mathcal{S}}V(\mathcal{J}_i))$, and only constantly many labelled copies on any fixed vertex set. Therefore, setting $\ell'=1$ in \eqref{eq:c1-c4} gives, for sufficiently large $n$, 
     $$\Delta_1(\mathcal{C}^{(\ell)})\le C_1n^{(\ell-1)(m-k)-1}\le \alpha D^{\ell-1}\log D,$$
     which verifies Lemma~\ref{lem:useful-lem} $(i)$. 
     %Since $\alpha>0$ is fixed, it follows that $\Delta_1(\mathcal C^{(\ell)})\le \alpha\frac{D^{\ell-1}}{\log D}$ for every $2\le \ell\le e$ and all sufficiently large $n$. This verifies Lemma~\ref{lem:useful-lem} (i). 
     Similarly, by the choice of $\gamma< 1/(2(m-k))$, for all $2\le\ell'<\ell$, we have
     $$\Delta_{\ell'}(\mathcal C^{(\ell)})\le C_1n^{(\ell-\ell')(m-k)-1}\le {D^{\ell-\ell'-\gamma}},$$ 
     which verifies Lemma~\ref{lem:useful-lem} $(ii)$.
     %$$\Delta_{\ell'}(\mathcal{C}^{(\ell)})\le C_1n^{(\ell-\ell')(m-k)-1}\le  D^{\ell-\ell'-\gamma},$$ which verifies Lemma~\ref{lem:useful-lem} $(ii)$. 
     Lastly, Lemma~\ref{lem:useful-lem} $(iii), (iv)$ hold trivially, since by \eqref{eq:obs} we have $\Delta_1(\mathcal{C}^{(2)})\le C_1n^{m-k-1}\le D^{1-\gamma}$. 

    Having verified that $\mathcal{H}$ and $\mathcal{C}$ satisfy all conditions of Lemma~\ref{lem:useful-lem}, we conclude that for sufficiently large $n$ there exists a $\mathcal{C}$-free matching $\mathcal{M}\subseteq\mathcal{H}$ satisfying $|\mathcal{M}|\ge \frac{|\mathcal{H}|}{D}(1-D^{-\alpha})$. Since $\mathcal{H}$ is $|\mathcal{J}|$-uniform and every vertex has degree at least $(1-d^{-\gamma})d$, we have $$|\mathcal{J}||\mathcal{H}|=\sum_{T\in V(\mathcal{H})}\deg_{\mathcal{H}}(T)\ge (1-d^{-\gamma})d|V(\mathcal{H})|.$$
    Hence, $$|\mathcal{M}|\ge \frac{1-d^{-\gamma}}{1+d^{-\gamma}}(1-D^{-\alpha})\frac{|V(\mathcal{H})|}{|\mathcal{J}|}\ge \frac{1-d^{-\gamma}}{1+d^{-\gamma}}(1-D^{-\alpha})\frac{(1-2\gamma)\binom{n}{k}}{|\mathcal{J}|}\ge (1-\ep)\frac{\binom{n}{k}}{|\mathcal{J}|},$$ where the last inequality holds for all sufficiently large $n$, using $\gamma<\frac{\ep}{8}$ and $d^{-\gamma},D^{-\alpha}=o(1)$, as needed. 

\section{Proof of Theorem~\ref{thm:main-union-free}, $a=0$}\label{sec:a=0}

\noindent In this section we prove the lower bound of $U_{t+1}(n,tk)$ in Theorem~\ref{thm:main-union-free}. Let $t\ge 2$ and $k\ge 2$ be fixed integers. It suffices to show that for every $\epsilon>0$, there exists an integer $n_0$ such that for all $n\ge n_0$,  
\begin{equation}\label{eq:a=0}
    U_{t+1}(n,tk)\ge(1-\epsilon)\cdot \frac{\binom{n}{k}}{\binom{tk-1}{k-1}}.
\end{equation}

\noindent Our proof is based on the following construction. 
Let $\mathcal{F}$ be a fixed $r$-graph. The {\it $k$-shadow} of $\mathcal{F}$, denoted by $\sigma_k(\mathcal{F})$, consists of all $k$-subsets of $V(\mathcal{F})$ that are contained in at least one member of $\mathcal{F}$; in other words, $\sigma_k(\mathcal{F})=\cup_{F\in\mathcal{F}}\binom{F}{k}$.

\begin{construction}\label{construction} 
    Let $\mathcal{F}$ be a fixed $r$-graph and $\sigma_k(\mathcal{F})$ be its $k$-shadow. Let $\mathcal{P}=\{(V(\mathcal{J}_i),\mathcal{J}_i):i\in[|\mathcal{P}|]\}$ be a {\it $\sigma_k(\mathcal{F})$-packing} in $\binom{[n]}{k}$, where the $\mathcal{J}_i$'s are pairwise edge-disjoint copies of $\sigma_k(\mathcal{F})$ with vertex set $V(\mathcal{J}_i)$. Put a copy $\mathcal{F}_i$ of $\mathcal{F}$ on top of each $\mathcal{J}_i$ such that $\sigma_k(\mathcal{F}_i)=\mathcal{J}_i$ and $V(\mathcal{F}_i)=V(\mathcal{J}_i)$, and let 
    \begin{align}\label{eq:H-sparse1}
        \mathcal{H}(\mathcal{F})=\bigcup_{i=1}^{|\mathcal{P}|} \mathcal{F}_i.   
    \end{align}
\noindent Then $\mathcal{H}(\mathcal{F})\subseteq\binom{[n]}{r}$ is a $r$-graph formed by the union of edge-disjoint copies of $\mathcal{F}$. 
\end{construction}

%Since the desired lower bound has order $\binom{n}{k}$, we use a near-optimal $\sigma_k(\mathcal{F})$-packing $\mathcal{P}$ in $\binom{[n]}{k}$ and form $\mathcal{H}(\mathcal{F})$ as above. \textcolor{red}{(This sentence seems to be redundant. What do you think?)}

Our strategy is to show that if $\mathcal{F}$ and $\mathcal{P}$ have some nice properties, then the $r$-graph (say, $r=tk$) $\mathcal{H}(\mathcal{F})$ above satisfies the following two properties: 
\begin{itemize}
    \item [(i)] $\mathcal{H}(\mathcal{F})$ is $(t+1)$-union-free;
    \item [(ii)] the number of edges $|\mathcal{H}(\mathcal{F})|=|\mathcal{P}|\cdot|\mathcal{F}|$ attains the required lower bound \eqref{eq:a=0}. 
\end{itemize}

\noindent The main difficulty lies in establishing the correct properties for $\mathcal{F}$ and $\mathcal{P}$.

The properties of $\mathcal{F}$ are captured by the following lemma.

\begin{lemma}\label{lem:uf-LB-F}
    Given integers $t\ge 2,~k\ge 2$ and a real $\epsilon\in(0,1)$, there exists an integer $m_0$ such that for every $m\ge m_0$ there exists a nonempty $tk$-graph $\mathcal{F}\subseteq\binom{[m]}{tk}$ satisfying the following properties:   
\begin{itemize}
    \item [{\rm (i)}] $\mathcal{F}$ is $(t+1)$-union-free;
    \item [{\rm (ii)}] $\mathcal{F}$ is $(\ell\cdot tk-(\ell-1)k-1,\ell)$-free for all $2\le\ell\le 2t+2$;
    \item [{\rm (iii)}] $\frac{|\mathcal{F}|}{|\sigma_k(\mathcal{F})|}\ge (1-\epsilon/2)\cdot\frac{1}{\binom{tk-1}{k-1}}$.
\end{itemize}
\end{lemma}

We now briefly explain the intuition behind Lemma \ref{lem:uf-LB-F}. A configuration violating $(t+1)$-union-freeness in $\mathcal{H}(\mathcal{F})$ may consist of edges from a single copy of $\mathcal{F}$ or from different copies. Property (i) excludes the first case, while property (ii), together with the local sparsity of $\mathcal{P}$, excludes the second. Property (iii) controls the number of edges in $\mathcal{H}(\mathcal{F})$. Combining it with a near-optimal $\sigma_k(\mathcal{F})$-packing $\mathcal{P}$ with $|\mathcal{P}|\ge(1-\epsilon/2)\cdot\binom{n}{k}/|\sigma_k(\mathcal{F})|$, we obtain
\begin{equation}\label{eq:lowerbound-a=0}
    |\mathcal{H}(\mathcal{F})|=|\mathcal{P}|\cdot|\mathcal{F}|\ge(1-\epsilon/2)\cdot\binom{n}{k}|\mathcal{F}|/|\sigma_k(\mathcal{F})|\ge(1-\epsilon)\cdot \frac{\binom{n}{k}}{\binom{tk-1}{k-1}},
\end{equation}
Thus condition (iii), together with the near-optimality of $\mathcal{P}$, gives the desired lower bound on $|\mathcal{H}(\mathcal{F})|$.

It remains to verify that the above properties of $\mathcal{F}$ and $\mathcal{P}$ imply that $\mathcal{H}(\mathcal{F})$ is $(t+1)$-union-free. This is accomplished by the following lemma.

\begin{lemma}\label{lem:a=0-H(F)}
    Let $\mathcal{F}\subseteq\binom{[m]}{tk}$ be a $tk$-graph satisfying Lemma~\ref{lem:uf-LB-F} (i), (ii), and let $\mathcal{P}=\{(V(\mathcal{J}_i),\mathcal{J}_i):i\in[|\mathcal{P}|]\}$ be an induced, $(2t+2)$-locally sparse $\sigma_k(\mathcal{F})$-packing. Then the $tk$-graph $\mathcal{H}(\mathcal{F})$ defined in Construction~\ref{construction} is $(t+1)$-union-free. 
\end{lemma}

\begin{proof}[Proof of Theorem~\ref{thm:main-union-free}, $a=0$]
    Given integers $t\ge 2,~k\ge 2$ and a real $\epsilon\in(0,1)$, choose $m$ sufficiently large and let $\mathcal{F}\subseteq\binom{[m]}{tk}$ be a $tk$-graph satisfying properties (i)--(iii) of Lemma~\ref{lem:uf-LB-F}. Since $\mathcal{F}$ contains a $tk$-edge, its $k$-shadow $\sigma_k(\mathcal{F})$ has at least $\binom{tk}{k}\ge2$ edges. Thus Lemma~\ref{lem:main-lemma} applies with $\mathcal{J}=\sigma_k(\mathcal{F})$, $e=2t+2$, and $\epsilon$ replaced by $\epsilon/2$. It yields, for all sufficiently large $n$, a near-optimal, induced, $(2t+2)$-locally sparse $\sigma_k(\mathcal{F})$-packing $\mathcal{P}$. Given the above $\mathcal{F}$ and $\mathcal{P}$, let $\mathcal{H}(\mathcal{F})$ be the $tk$-graph defined in Construction~\ref{construction}. 
    By Lemma~\ref{lem:a=0-H(F)}, $\mathcal{H}(\mathcal{F})$ is $(t+1)$-union-free. Moreover, \eqref{eq:lowerbound-a=0} shows that $|\mathcal{H}(\mathcal{F})|$ meets the lower bound required in \eqref{eq:a=0}. This completes the proof for the case $a=0$ in Theorem~\ref{thm:main-union-free}. %$a=0$ in Theorem~\ref{thm:main-union-free}. 
\end{proof} 

We postpone the proofs of Lemmas \ref{lem:uf-LB-F} and \ref{lem:a=0-H(F)} to the next two subsections.

\subsection{Proof of Lemma~\ref{lem:uf-LB-F}}

\noindent We construct the required $tk$-graph $\mathcal{F}$ as follows. Let $c=c(tk-1,k,2t+2)$ be the constant obtained by applying Lemma~\ref{lem:e^--free} with $r=tk-1$ and $e=2t+2$. Choose $q_0=q_0(t,k,\epsilon)$ sufficiently large such that, for every $q\ge q_0$,
$$\frac{\binom{q}{k}}{\lceil c(q-1)^{k+\frac{1}{2t+1}}\rceil}\le \frac{\epsilon}{2}\binom{tk-1}{k-1}.$$
Set $m_0=q_0+\lceil c q_0^{k+\frac{1}{2t+1}}\rceil$ and fix any $m\ge m_0$. Let $q$ be the smallest positive integer satisfying $q+\lceil c q^{k+\frac{1}{2t+1}}\rceil\ge m$, and $N=m-q$. Then $q\ge q_0$. By the minimality of $q$, we have $q-1+\lceil c(q-1)^{k+\frac{1}{2t+1}}\rceil<m$, which implies $$\lceil c(q-1)^{k+\frac{1}{2t+1}}\rceil\le N\le \lceil c q^{k+\frac{1}{2t+1}}\rceil.$$

Let $\mathcal{G}\subseteq\binom{[q]}{tk-1}$ be a $(tk-1)$-graph obtained from Lemma~\ref{lem:e^--free}. Then $\mathcal{G}$ has at least $cq^{k+\frac{1}{2t+1}}$ edges and is $(\ell(tk-1)-(\ell-1)k-1,\ell)$-free for all $2\le\ell\le 2t+2$. Since $|\mathcal{G}|\ge \lceil c q^{k+\frac{1}{2t+1}}\rceil\ge N$, we may select $N$ distinct edges $G_1,\ldots,G_N$ from $\mathcal{G}$. Introduce a set $U=\{u_i:i\in[N]\}$ of new vertices disjoint from $[q]$, and identify $[q]\cup U$ with $[m]$. Define
 $$\mathcal{F}=\{G_i\cup \{u_i\}:i\in[N]\}\subseteq\binom{[m]}{tk}.$$
The family $\mathcal{F}$ is $(t+1)$-union-free, indeed $s$-union-free for every $s\ge1$, because the private vertex $u_i$ appears only in the edge $G_i \cup \{u_i\}\in\mathcal{F}$. Thus the set of private vertices contained in a union determines exactly which edges of $\mathcal{F}$ participate in that union.

    To prove (ii), take any $I\subseteq[N]$ with $|I|=\ell$, where $2\le\ell\le 2t+2$. Since the private vertices are distinct and outside $[q]$, and since $\mathcal{G}$ is $(\ell(tk-1)-(\ell-1)k-1,\ell)$-free, we have 
    $$\left|\bigcup_{i\in I}(G_i\cup \{u_i\})\right|=\left|\bigcup_{i\in I}G_i\right|+\ell\ge \ell(tk-1)-(\ell-1)k+\ell=\ell tk-(\ell-1)k.$$
    This proves that $\mathcal{F}$ is $(\ell tk-(\ell-1)k-1,\ell)$-free for every $2\le \ell\le 2t+2$.

    Finally, the $k$-shadow of $\mathcal{F}$ consists of $k$-sets entirely in the base $[q]$ and $k$-sets containing one private vertex. Therefore, $|\sigma_k(\mathcal{F})|\le\binom{q}{k}+|\mathcal{F}|\binom{tk-1}{k-1}$. Using $|\mathcal{F}|=N$ and the choice of $q_0$, we obtain
    $$\frac{|\mathcal{F}|}{|\sigma_k(\mathcal{F})|}\ge \frac{N}{\binom{q}{k}+N\binom{tk-1}{k-1}}=\frac{1}{\binom{q}{k}/N+\binom{tk-1}{k-1}}\ge \frac{1}{(1+\epsilon/2)\binom{tk-1}{k-1}}\ge (1-\epsilon/2)\frac{1}{\binom{tk-1}{k-1}}.$$
    This proves (iii), and hence the lemma.

\subsection{Proof of Lemma~\ref{lem:a=0-H(F)}}

\noindent Before presenting the proof of Lemma~\ref{lem:a=0-H(F)}, it will be convenient to state a useful claim.

\begin{claim}\label{claim:induced}
    Let $\mathcal{F}$ be a $tk$-graph and $\mathcal{P}=\{(V(\mathcal{J}_i),\mathcal{J}_i):i\in[|\mathcal{P}|]\}$ be an induced $\sigma_k(\mathcal{F})$-packing. Let $\mathcal{H}(\mathcal{F})$ be defined as in Construction~\ref{construction}. Then the following hold:
    \begin{itemize}
        \item [(i)] Suppose that $\mathcal{F}_i$ and $\mathcal{F}_{i'}$ are two distinct copies of $\mathcal{F}$ in $\mathcal{H}(\mathcal{F})$. Then for every edge $F\in\mathcal{F}_i$, we have $|F\cap V(\mathcal{F}_{i'})|\le k-1$.
        \item [(ii)] If $\mathcal{F}$ is $(2tk-k-1,2)$-free, then for every distinct $F,F'\in\mathcal{H}(\mathcal{F})$, we have $|F\cap F'|\le k$.
    \end{itemize}
\end{claim}

\begin{proof}
    To prove (i), first observe from Construction~\ref{construction} that $V(\mathcal{F}_{i})=V(\mathcal{J}_{i})$ and $V(\mathcal{F}_{i'})=V(\mathcal{J}_{i'})$. Now, if $|F\cap V(\mathcal{F}_{i'})|\ge k$, then by construction there exists a $k$-subset $T\subseteq F\subseteq V(\mathcal{J}_{i})$ such that $T\subseteq V(\mathcal{J}_{i})\cap V(\mathcal{J}_{i'})$. Since $T\in\binom{F}{k}\subseteq \sigma_k(\mathcal{F}_{i})=\mathcal{J}_{i}$, this contradicts the assumption that $\mathcal{P}$ is induced.

    The proof of (ii) splits into two cases. If $F$ and $F'$ belong to the same copy of $\mathcal{F}$, then the $(2tk-k-1,2)$-freeness of $\mathcal{F}$ directly implies $|F \cap F'| \le k$. If $F$ and $F'$ lie in different copies, say $F \in \mathcal{F}_i$ and $F' \in \mathcal{F}_{i'}$ with $i \neq i'$, then by (i), $|F \cap F'| \le |F \cap V(\mathcal{F}_{i'})| \le k-1$. In either case, the required inequality holds.
\end{proof}

We are now in a position to prove Lemma~\ref{lem:a=0-H(F)}. 

\begin{proof}[Proof of Lemma~\ref{lem:a=0-H(F)}]
    Let us start by showing that $\mathcal{H}(\mathcal{F})$ is $(\ell\cdot tk-(\ell-1)k-1,\ell)$-free for all $2\le\ell\le 2t+2$. Fix such an $\ell$ and let $F_1, \ldots, F_{\ell}$ be any $\ell$ distinct edges of $\mathcal{H}(\mathcal{F})$. These edges belong to some set of distinct copies $\mathcal{F}_1, \ldots, \mathcal{F}_s$ of $\mathcal{F}$, for some $1\le s\le 2t+2$. For each $i \in [s]$, let $L_i \subseteq [\ell]$ consist of those indices $j$ for which $F_j \in \mathcal{F}_i$, and write $\ell_i = |L_i|$. Observe that the family $\{L_i : i \in [s]\}$ partitions $[\ell]$, and hence $\sum_{i=1}^s \ell_i = \ell$. 
    
    For each $i\in[s]$, let $X_i = \cup_{j \in L_i} V(F_j)$. Since $X_i\subseteq V(\mathcal{F}_i)$, we have $$\left|\bigcup_{i=1}^sV(\mathcal{F}_i)\right|\le\left|\bigcup_{i=1}^sX_i\right|+\left(\sum_{i=1}^s|V(\mathcal{F}_i)|-\sum_{i=1}^s|X_i|\right).$$ 
    Using the freeness of $\mathcal{F}$ from Lemma~\ref{lem:uf-LB-F}(ii), we have $|X_i|\ge \ell_i\cdot tk-(\ell_i-1)k$ for every $i\in[s]$ (with the case $\ell_i=1$ being trivial). Using the local sparsity of $\mathcal{P}$, we also have $|\bigcup_{i=1}^sV(\mathcal{F}_i)|\ge sm-(s-1)k$ (again trivial for $s=1$). Since each $X_i\subseteq V(\mathcal{F}_i)$, these two inequalities give 
    \begin{align*}
        \left|\bigcup_{i=1}^\ell F_i\right|=\left|\bigcup_{i=1}^s X_i\right|
        &\ge \sum_{i=1}^s|X_i|+\left|\bigcup_{i=1}^sV(\mathcal{F}_i)\right|-\sum_{i=1}^s|V(\mathcal{F}_i)|\\
        &\ge \sum_{i=1}^s\left(\ell_i\cdot tk-(\ell_i-1)k\right)+\left(sm-(s-1)k\right)-sm\\
        &=\ell\cdot tk-(\ell-1)k,
    \end{align*}
    as needed. 
    
    Next, we prove that $\mathcal{H}(\mathcal{F})$ is $t$-cover-free. Suppose for contradiction that there exist $t+1$ distinct edges $F_1,\ldots,F_{t+1}\in\mathcal{H}(\mathcal{F})$ with $F_{t+1}\subseteq \cup_{i=1}^{t}F_i$. This implies that
    \begin{equation}\label{eq:t*tk}
        \left|\bigcup_{i=1}^{t+1}F_i\right|=\left|\bigcup_{i=1}^{t}F_i\right|\le t\cdot tk.
    \end{equation}

    \noindent If the inequality \eqref{eq:t*tk} is strict, then we obtain a contradiction to the $((t+1)\cdot tk-tk-1,t+1)$-freeness of $\mathcal{H}(\mathcal{F})$.
    If equality holds, then $F_1,\ldots,F_t$ are pairwise disjoint. Moreover, by Claim~\ref{claim:induced}(ii), $|F_{t+1}\cap F_i|\le k$ for all $i\in [t]$. Since $F_{t+1}\subseteq \cup_{i=1}^{t}F_i$, it follows that $|F_{t+1}\cap F_i|= k$ for every $i$. Claim~\ref{claim:induced}(i) then implies that $F_1,\ldots,F_{t+1}$ are contained in the same copy of $\mathcal{F}$, contradicting the fact that $\mathcal{F}$ is $t$-cover-free, which follows from the assumption that $\mathcal{F}$ is $(t+1)$-union-free. 

    Finally, we prove that $\mathcal{H}(\mathcal{F})$ is $(t+1)$-union-free. Assume for contradiction that there exist distinct subfamilies $\mathcal{A},\mathcal{B}\subseteq \mathcal{H}(\mathcal{F})$, with $1\le |\mathcal{A}|,|\mathcal{B}|\le t+1$, such that $\cup_{A\in \mathcal{A}}A=\cup_{B\in \mathcal{B}}B$. Since $\mathcal{H}(\mathcal{F})$ is $t$-cover-free, we have $|\mathcal{A}|=|\mathcal{B}|=t+1$.
    Without loss of generality, write
    $$\mathcal{A}=\{C_1,\ldots,C_p,A_{p+1},\ldots, A_{t+1}\},\mathcal{B}=\{C_1,\ldots,C_p,B_{p+1},\ldots,B_{t+1}\},$$
    where $\mathcal{A}\cap\mathcal{B}=\{C_1,\ldots,C_p\}$ and $0\le p\le t$. Let $X=\cup_{i=1}^pC_i$, $Y=\cup_{i=p+1}^{t+1}A_i$, and $Z=\cup_{i=p+1}^{t+1}B_i$. From $X\cup Y=X\cup Z$, we obtain
    \begin{align}\label{eq-union-free-upper}
        |X\cup Y\cup Z|=|X\cup Y|=|X\cup Z|\le (t+1)\cdot tk.
    \end{align}
    
    \noindent Observe that $\mathcal{A}\cup\mathcal{B}$ consists of $2t+2-p$ distinct edges. Applying the $((2t+2-p)\cdot tk-(2t+1-p)k-1,2t+2-p)$-freeness of $\mathcal{H}(\mathcal{F})$ yields
    \begin{align}\label{eq-union-free-lower}
        |X\cup Y\cup Z|\ge (2t+2-p)\cdot tk-(2t+1-p)k.
    \end{align}

    \noindent Together with the upper bound \eqref{eq-union-free-upper}, we obtain 
    $$(t+1)\cdot tk\ge (2t+2-p)\cdot tk-(2t+1-p)k.$$ 
    \noindent Simplifying this inequality leads to $t-\frac{1}{t-1}\le p\le t$. Given that $p$ is an integer between $0$ and $t$, the only feasible cases are:
    (1) $t\ge 2$ and $p=t$; (2) $t=2$ and $p=1$. We will treat these two cases separately.
    
    \paragraph{Case (1): $t\ge 2$ and $p=t$.} Then $Y=A_{t+1}$ and $Z=B_{t+1}$. By Claim~\ref{claim:induced} (ii), write $|A_{t+1}\cap B_{t+1}|=k-x$ with $x\ge 0$. The condition $Y\Delta Z=A_{t+1}\Delta B_{t+1}\subseteq X$ implies 
    $$|X\cup A_{t+1}\cup B_{t+1}|\le |X|+|A_{t+1}\cap B_{t+1}|\le t\cdot tk + k-x.$$
    Since $\mathcal{H}(\mathcal{F})$ is $((t+2)\cdot tk-(t+1)k-1,\ t+2)$-free, we also have 
    $$|X\cup A_{t+1}\cup B_{t+1}|\ge (t+2)\cdot tk-(t+1)k.$$
    Combining the two inequalities above gives $(t+2)\cdot tk-(t+1)k\le t\cdot tk + k-x$, which holds only if $t=2$ and $x=0$. 
    
    Now $\mathcal{H}(\mathcal{F})$ is a $2k$-graph, and we have $C_1\cup C_2\cup A_3=C_1\cup C_2\cup B_3$ with $|A_3\cap B_3|=k$. By Claim~\ref{claim:induced}(i), $A_3$ and $B_3$ belong to the same copy of $\mathcal{F}$, say $\mathcal{F}_i$. Since $\mathcal{F}_i$ is $3$-union-free, at least one of $C_1$ and $C_2$ does not belong to  $\mathcal{F}_i$; assume $C_1\notin \mathcal{F}_i$. Then Claim~\ref{claim:induced}(i) yields $|C_1\cap (A_3\Delta B_3)|\le |C_1\cap V(\mathcal{F}_i)|\le k-1$. As $A_3\Delta B_3\subseteq X=C_1\cup C_2$ and $|A_3\Delta B_3|=2k$, we have 
    $$|C_2\cap (A_3\Delta B_3)|\ge|A_3\Delta B_3|-|C_1\cap (A_3\Delta B_3)|\ge 2k-(k-1)=k+1.$$
    Therefore, $$|C_2\cup A_3\cup B_3|\le|C_2|+|A_3\cup B_3|-|C_2\cap (A_3\Delta B_3)|\le 4k-1=3\cdot 2k-2k-1,$$ 
    contradicting the $(3\cdot 2k-2k-1,3)$-freeness of $\mathcal{H}(\mathcal{F})$. 
    
    \paragraph{Case (2): $t=2$ and $p=1$.} Now we have $C_1\cup A_2\cup A_3=C_1\cup B_2\cup B_3$. By \eqref{eq-union-free-upper}, we have 
    $$|C_1\cup A_2\cup A_3\cup B_2\cup B_3|\le 6k.$$
    Since $\mathcal{H}(\mathcal{F})$ is $(5\cdot 2k-4k-1,5)$-free, we also have 
    $$|C_1\cup A_2\cup A_3\cup B_2\cup B_3|\ge 5\cdot 2k-4k=6k.$$
    Consequently, 
    $$|C_1\cup A_2\cup A_3\cup B_2\cup B_3|=|C_1\cup A_2\cup A_3|=|C_1\cup B_2\cup B_3|=6k.$$
    Given that $\mathcal{H}(\mathcal{F})$ is a $2k$-graph, the above equation implies that $C_1,A_2,A_3$ are pairwise disjoint, and so are $C_1,B_2,B_3$. This forces $A_2\cup A_3=B_2\cup B_3$. Therefore, 
    $$|A_2\cup A_3\cup B_2\cup B_3|=|A_2\cup A_3|\le 4k,$$ 
    contradicting the $(4\cdot 2k-3k-1,4)$-freeness of $\mathcal{H}(\mathcal{F})$. 

    \vspace{10pt}

    Therefore, $\mathcal{H}(\mathcal{F})$ is $(t+1)$-union-free. This completes the proof.
\end{proof}

\section{Concluding remarks}\label{sec:conclusion}

\noindent We have established the asymptotics of $U_{t+1}(n,r)$, up to a lower-order term, for almost all integers $t \ge 2$ and $r \ge 3$. The only families not covered are $\{U_3(n,2s-1), U_{s+1}(n,s+1), U_{s+1}(n,s+2): s \ge 2\}$ and $U_4(n,7)$. Our work leaves several natural open questions.

\paragraph{The exceptional cases.} For the parameters where our construction does not apply, the known lower and upper bounds do not match, leaving a gap in the exponent of $n$. It is interesting to determine the Tur\'an exponent or its non-existence for these $U_{t+1}(n,r)$. In relation to this, we recall that F\"{u}redi and Ruszink\'{o} \cite{Furedi-Ruszinko-no-grid} proved $U_r(n,r) > n^{2-o(1)}$ for all $r \ge 3$ and conjectured $U_r(n,r) = o(n^2)$. If their conjecture is true, then $U_r(n,r)$ does not have a Tur\'an exponent.   

\paragraph{The case $U_2(n,r)$.} While our main theorem does not provide sharp asymptotic bounds for $U_2(n,r)$, we can improve the lower bound when $3$ divides $r$. For any $k \ge 1$, we can show $U_2(n,3k) \ge (1-o(1))\binom{n}{2k}/\binom{3k}{2k}$, which improves the longstanding lower bound $U_2(n,3k)\ge\left(1/(3k)^{(3k)}-o(1)\right)n^{2k}$ of Frankl and F\"uredi (see \cite[Proposition 1.2 and Theorem 1.3]{Frankl-Furedi-2-union-free}). Since the proof follows a strategy analogous to that of Theorem~\ref{thm:main-union-free}, we omit its details. A fundamental question is whether the degenerate Tur\'{a}n density $\lim_{n \to \infty} U_2(n,r) / n^{\lceil 4r/3 \rceil / 2}$ exists. For $r=2$, this question reduces to the classic conjecture of Erd\H{o}s \eqref{eq:erdos-conjecture}.

\paragraph{Deterministic constructions.} Our results rely on probabilistic methods. The problem of finding explicit (deterministic) constructions that match these asymptotic bounds merits further investigation.

\section*{Acknowledgments}
\noindent M. Liu, C. Shangguan, and C. Zhang are supported by the National Natural Science Foundation of China under Grant Nos. 12571352 and 12231014, and the Fundamental Research Funds for the Central Universities. M. Liu is also supported by the China Scholarship Council and Institute for Basic Science IBS-R029-C4. 

{\small
\normalem
\bibliographystyle{plain}
\bibliography{Article}}

\appendix

\section{Proof of Theorem~\ref{thm:main-union-free}, $1\le a\le t-1$}\label{sec:a>0}

\noindent  This section proves the lower bound on $U_{t+1}(n,tk-a)$ for $1\le a\le t-1$ stated in Theorem~\ref{thm:main-union-free}. Let $t\ge 2$, $k\ge 2$, and $1\le a \le t-1$ be fixed integers. It suffices to prove that, apart from the cases $U_3(n,2k-1)$, $U_{t+1}(n,t+1)$, $U_{t+1}(n,t+2)$, and $U_4(n,7)$, the inequality
\begin{equation}\label{eq:a>0}
    U_{t+1}(n,tk-a)\ge(1-\epsilon)\cdot \frac{\binom{n}{k}}{\binom{tk-a}{k}-m(tk-a,k,t-a-1)}
\end{equation}
holds for every $\epsilon>0$ and sufficiently large $n$.

Our proof relies on the following construction. Let $\mathcal{G}\subseteq\binom{[tk-a]}{k}$ be a $k$-graph on $(tk-a)$-vertices that contains no $t-a$ pairwise disjoint edges and has the maximum number of edges, i.e., $|\mathcal{G}|=m(tk-a,k,t-1-a)$.
Let 
\begin{equation}\label{eq:F:a>0}
    \mathcal{F}=\binom{[tk-a]}{k}\setminus \mathcal{G}.
\end{equation}
To apply Lemma~\ref{lem:main-lemma} with $\mathcal{J}=\mathcal{F}$, we need to verify that $|\mathcal{F}|\ge2$. If $a=t-1$, then $\mathcal{G}=\emptyset$ and hence $\mathcal{F}=\binom{[tk-a]}{k}$, which contains at least two edges. Now suppose that $1\le a\le t-2$. If $|\mathcal{F}|\le1$, then there is a $k$-set $K$ such that $\binom{[tk-a]}{k}\setminus\{K\}\subseteq\mathcal{G}$. We claim that $\binom{[tk-a]}{k}\setminus \{K\}$ contains $t-a$ pairwise disjoint edges. Indeed, since $|[tk-a]\setminus K|=(t-a-1)k+a(k-1)$ and $a(k-1)\ge1$, we can choose $t-a-1$ pairwise disjoint $k$-sets in $[tk-a]\setminus K$, with at least one vertex $v\in[tk-a]\setminus K$ remaining. Let $K'\subseteq K$ be any $(k-1)$-subset. Then $K'\cup{v}$ is a $k$-set disjoint from the previously chosen edges and distinct from $K$. Thus we obtain $t-a$ pairwise disjoint edges, all of which belong to $\mathcal{G}$, contradicting the defining property of $\mathcal{G}$. Thus $|\mathcal{F}|\ge2$. Let $\mathcal{P}=\{(V(\mathcal{F}_i),\mathcal{F}_i):i\in[|\mathcal{P}|]\}$ be a near-optimal $\mathcal{F}$-packing satisfying $|\mathcal{P}|\ge(1-\epsilon)\cdot\binom{n}{k}/|\mathcal{F}|$, and define the $(tk-a)$-graph $\mathcal{H}:=\{V(\mathcal{F}_i): 1\le i\le |\mathcal{P}|\}$. Since $|\mathcal{F}|=\binom{tk-a}{k}-m(tk-a,k,t-a-1)$, we obtain $$|\mathcal{H}|\ge(1-\epsilon)\cdot \frac{\binom{n}{k}}{\binom{tk-a}{k}-m(tk-a,k,t-a-1)},$$ which is exactly the bound required in \eqref{eq:a>0}. It remains to require $\mathcal{P}$ to satisfy additional properties and to prove that $\mathcal{H}$ is $(t+1)$-union-free. The construction follows the same general idea as Construction~\ref{construction} for $a=0$, but is technically simpler because shadows are not involved.
 
The following lemma plays the same role for $a>0$ as Lemma~\ref{lem:a=0-H(F)} does for $a=0$.

\begin{lemma}\label{lem:union-free-t>2}
    Let $t\ge 2$, $k\ge 2$, $1\le a \le t-1$ be integers such that $$(t+1,tk-a)\not\in\{(3,2s-1),(s+1,s+1),(s+1,s+2):s\ge 2\}\cup\{(4,7)\}.$$
    Let $\mathcal{F}\subseteq\binom{[tk-a]}{k}$ be the $k$-graph defined in \eqref{eq:F:a>0} and $\mathcal{P}=\{(V(\mathcal{F}_i),\mathcal{F}_i):~i\in[|\mathcal{P}|]\}$ be an induced, $(2t+2)$-locally sparse $\mathcal{F}$-packing. Then the $(tk-a)$-graph $\mathcal{H}:=\{V(\mathcal{F}_i):~i\in[|\mathcal{P}|]\}$ is $(t+1)$-union-free.
\end{lemma}

\begin{proof}\label{pf:general case}
    We first prove that $\mathcal{H}$ is $t$-cover-free. Suppose for contradiction that there exist $t+1$ $\mathcal{F}_i$'s such that $V(\mathcal{F}_{t+1})\subseteq \cup_{i=1}^tV(\mathcal{F}_i)$. Clearly, the sets $T_1=V(\mathcal{F}_{1}) \cap V(\mathcal{F}_{t+1})$ and 
    $$T_i=V(\mathcal{F}_{i})\cap \left(V(\mathcal{F}_{t+1})~\bigg \backslash~\left(\bigcup_{j=1}^{i-1}V(\mathcal{F}_{j})\right)\right),~2\le i\le t,$$
    
    \noindent are pairwise disjoint. Since $\mathcal{P}$ is induced, each $T_i$ has size at most $k$. Moreover, since $V(\mathcal{F}_{t+1})\subseteq \cup_{i=1}^t T_i$ and $|V(\mathcal{F}_{i})|=tk-a=(t-a)k+a(k-1)$, at least $t-a$ of the $T_i$ have size exactly $k$.
    Assume without loss of generality that $|T_1|=\cdots=|T_{t-a}|=k$. The induced property of $\mathcal{P}$ further implies that for every $1\le i\le t-a$, $T_i=\binom{V(\mathcal{F}_{t+1})}{k}\cap\binom{V(\mathcal{F}_i)}{k}$ and $T_i\notin \mathcal{F}_{t+1}$. Therefore, $$\{T_1,\ldots,T_{t-a}\}\subseteq\binom{V(\mathcal{F}_{t+1})}{k}\backslash\mathcal{F}_{t+1}\simeq\mathcal{G}.$$ 
    Since $T_1,\ldots,T_{t-a}$ are pairwise disjoint, we obtain $t-a$ pairwise disjoint edges in $\mathcal{G}$, a contradiction. This proves that $\mathcal{H}$ is $t$-cover-free.

    We proceed to show that $\mathcal{H}$ is $(t+1)$-union-free. 
    Assume for contradiction that there exist distinct $\mathcal{A},\mathcal{B}\subseteq \mathcal{H}$, with $1\le |\mathcal{A}|,|\mathcal{B}|\le t+1$, such that $\cup_{A\in \mathcal{A}}A=\cup_{B\in \mathcal{B}}B$. Since $\mathcal{H}$ is $t$-cover-free, we have $|\mathcal{A}|=|\mathcal{B}|=t+1$.
    Without loss of generality, write 
    $$\mathcal{A}=\{C_1,\ldots,C_p,A_{p+1},\ldots, A_{t+1}\},\mathcal{B}=\{C_1,\ldots,C_p,B_{p+1},\ldots,B_{t+1}\},$$
    where $\mathcal{A}\cap\mathcal{B}=\{C_1,\ldots,C_p\}$ and $0\le p\le t$. Denote $X=\cup_{i=1}^pC_i$, $Y=\cup_{i=p+1}^{t+1}A_i$, and $Z=\cup_{i=p+1}^{t+1}B_i$. Then $X\cup Y=X\cup Z$. Applying the $((2t+2-p)\cdot (tk-a)-(2t+1-p)k-1,2t+2-p)$-freeness of $\mathcal{H}$, we obtain 
    $$(2t+2-p)\cdot (tk-a)-(2t+1-p)k\le |X\cup Y\cup Z|=|X\cup Y|\le (t+1)\cdot (tk-a).$$
    Simplifying this inequality yields 
    \begin{equation}\label{eq:p-a>0}
        t-\fr{a+k}{tk-a-k}\le p\le t,
    \end{equation}
    which holds only when either $p=t$, or $p=t-1$ with $(t,k,a)\in\{(3,4,2),(4,3,3)\}$. In the latter case, one uses $a+k\ge tk-a-k$ together with the standing assumption that $(t+1,tk-a)\not\in\{(3,2s-1),(s+1,s+1),(s+1,s+2):s\ge 2\}\cup\{(4,7)\}$. 
    \noindent We analyze these two cases separately.
    
    \paragraph{Case (1): $p=t$.} Here $Y=A_{t+1}$ and $Z=B_{t+1}$. Since both $A_{t+1}$ and $B_{t+1}$ are vertex sets of $\mathcal{F}$-copies in $\mathcal{P}$, the induced property of $\mathcal{P}$ gives that $|A_{t+1}\cap B_{t+1}|\le k$.
    Since $Y\Delta Z=A_{t+1}\Delta B_{t+1}\subseteq X$, we have $|X\cup A_{t+1}\cup B_{t+1}|\le |X|+|A_{t+1}\cap B_{t+1}|$. Moreover, using the fact that $\mathcal{H}$ is $((t+2)\cdot(tk-a)-(t+1)k-1,t+2)$-free, we obtain
    \begin{equation}\label{eq:p=t}
        (t+2)\cdot (tk-a)-(t+1)k\le |X\cup A_{t+1}\cup B_{t+1}|\le |X|+|A_{t+1}\cap B_{t+1}|\le t\cdot (tk-a) + k.
    \end{equation}
    Simplifying \eqref{eq:p=t} gives $tk-2k-2a\le0$, which occurs only if $(t,k,a)\in\{(3,4,2),(4,3,3)\}$. In both cases we have $a=t-1$, and hence $|\mathcal{G}|=m(tk-a,k,0)=0$. So $\mathcal{F}=\binom{[tk-a]}{k}$. Under this condition, the induced property of $\mathcal{P}$ implies that for every pair of distinct $\mathcal{F}$-copies $\mathcal{F}_i$ and $\mathcal{F}_j$ in $\mathcal{P}$, $|V(\mathcal{F}_i) \cap V(\mathcal{F}_j)| \le k-1$. This implies the stronger inequality $|A_{t+1}\cap B_{t+1}|\le k-1$. Substituting this into \eqref{eq:p=t} yields $$(t+2)\cdot (tk-a)-(t+1)k-(t\cdot(tk-a) + k-1)\le0.$$
    However, for $(t,k,a)\in\{(3,4,2),(4,3,3)\}$, the left-hand side is positive, a contradiction.
    
    \paragraph{Case (2): $p=t-1$.} Suppose first that $(t,k,a)=(3,4,2)$. Then $C_1\cup C_2\cup A_3\cup A_4=C_1\cup C_2\cup B_3\cup B_4$. Hence,
    $$|C_1\cup C_2\cup A_3\cup A_4\cup B_3\cup B_4|=|C_1\cup C_2\cup A_3\cup A_4|\le (t+1)(tk-a)=40.$$
    On the other hand, since $\mathcal{H}$ is $(6\cdot 10-5\cdot4-1,6)$-free, we also have 
    $$|C_1\cup C_2\cup A_3\cup A_4\cup B_3\cup B_4|\ge 40.$$
    Thus we have equality in both bounds. Since $\mathcal{H}$ is a $10$-graph, $C_1,C_2,A_3,A_4$ have to be pairwise disjoint. Similarly, $C_1,C_2,B_3,B_4$ are also pairwise disjoint. Consequently, $A_3\cup A_4=B_3\cup B_4$, and hence $$|A_3\cup A_4\cup B_3\cup B_4|=|A_3\cup A_4|= 20,$$ which contradicts the $(4\cdot10-3\cdot4-1,4)$-freeness of $\mathcal{H}$. 
    
    If $(t,k,a)=(4,3,3)$, then $C_1\cup C_2\cup C_3\cup A_4\cup A_5=C_1\cup C_2\cup C_3\cup B_4\cup B_5$. Hence, 
    $$|C_1\cup C_2\cup C_3\cup A_4\cup A_5\cup B_4\cup B_5|=|C_1\cup C_2\cup C_3\cup A_4\cup A_5|\le (t+1)(tk-a)=45.$$
    The $(7\cdot9-6\cdot3-1,7)$-freeness of $\mathcal{H}$ yields the reverse inequality 
    $$|C_1\cup C_2\cup C_3\cup A_4\cup A_5\cup B_4\cup B_5|\ge 45.$$ 
    Again, equality holds in both bounds, implying that $A_4\cup A_5=B_4\cup B_5$. Consequently, $$|A_4\cup A_5\cup B_4\cup B_5|=|A_4\cup A_5|=18,$$
    which contradicts the $(4\cdot9-3\cdot3-1,4)$-freeness of $\mathcal{H}$. %This completes the analysis of Case (2).
\end{proof}

    Combining Lemmas \ref{lem:main-lemma} and \ref{lem:union-free-t>2}, we now prove the lower bound of Theorem~\ref{thm:main-union-free} for $1\le a\le t-1$.
    
    \begin{proof}[Proof of Theorem~\ref{thm:main-union-free} for $1\le a \le t-1$]
        Let $t\ge 2$, $k\ge 2$, $1\le a \le t-1$ be integers satisfying the assumption of Theorem~\ref{thm:main-union-free}. Let $\mathcal{F}$ be the $k$-graph defined in \eqref{eq:F:a>0}. As shown above, $|\mathcal{F}|\ge2$, and also $tk-a>k$. Thus Lemma~\ref{lem:main-lemma} applies with $\mathcal{J}=\mathcal{F}$ and $e=2t+2$. For any $\epsilon\in(0,1)$ and sufficiently large $n$, let $\mathcal{P}$ be the induced, $(2t+2)$-locally sparse, near-optimal packing obtained in this way. %by applying Lemma~\ref{lem:main-lemma} with $\mathcal{J}=\mathcal{F}$ and $e=2t+2$. 
        Then it follows from Lemma~\ref{lem:main-lemma} (i), (ii) and Lemma~\ref{lem:union-free-t>2} that $\mathcal{H}=\{V(\mathcal{F}_i): 1\le i\le |\mathcal{P}|\}\subseteq \binom{[n]}{tk-a}$ is $(t+1)$-union-free. Moreover, by Lemma~\ref{lem:main-lemma} (iii) and \eqref{eq:F:a>0} we have 
        $$|\mathcal{H}|=|\mathcal{P}|\ge(1-\epsilon)\frac{1}{|\mathcal{F}|}\cdot \binom{n}{k}=(1-\epsilon)\cdot \frac{1}{\binom{tk-a}{k}-m(tk-a,k,t-a-1)}\cdot \binom{n}{k},$$
        which proves the desired lower bound \eqref{eq:a>0}.
    \end{proof}

\section{Proof of Claim~\ref{claim-concentration}}\label{sec:claim}

\noindent To prove Claim~\ref{claim-concentration}, we need the following bounded differences inequality, also known as McDiarmid's inequality or the Hoeffding--Azuma inequality.

\begin{lemma}[see Theorem 9.1.3 in \cite{zhaoyufei22prob}]\label{lem:BDI}
    Let $\Omega_1,\ldots,\Omega_s$ be finite sets and $X_1\in\Omega_1,\ldots,X_s\in\Omega_s$ be independent random variables. Suppose that $f:\Omega_1\times\cdots\times\Omega_s\to\mathbb{R}$ satisfies
    $$|f(x_1,\ldots,x_s)-f(x_1',\ldots,x_s')|\le c_i,$$
    whenever $(x_1,\ldots,x_s)$ and $(x_1',\ldots,x_s')$ differ only on the $i$-th coordinate. Here $c_1,\ldots,c_s$ are non-negative constants. Then, the random variable $Z=f(X_1,\ldots,X_s)$ satisfies for every $\lambda\ge0$,
    $$\Pr[|Z-\mathbb{E}[Z]|\ge\lambda]\le2\exp{\left(\frac{-2\lambda^2}{c_1^2+\cdots+c_s^2}\right)}.$$
\end{lemma}

\begin{proof}[Proof of Claim~\ref{claim-concentration}]
    For each $K\in V(\mathcal{H})\subseteq \binom{[n]}{k}$, let $\mathcal{J}(K)$ be the multiset of all labelled copies $\mathcal{J}'$ of $\mathcal{J}$ in $\binom{[n]}{k}$ that contain $K$. %, i.e., $$\mathcal{J}(K)=\left\{\mathcal{J}'\subseteq \binom{[n]}{k}: \mathcal{J}'\text{ is a copy of } \mathcal{J},~K\in \mathcal{J}'\right\}.$$
    Distinct labelled copies are counted separately, even if they have the same underlying vertex set as parallel edge of $\mathcal H$. In particular, cardinalities and sums involving $\mathcal{J}(K)$ are always taken with multiplicity. The vertex set $V(\mathcal{J}')$ of any $\mathcal{J}'\in \mathcal{J}(K)$ contains $K$, while the remaining $m-k$ vertices can be chosen arbitrarily from $[n]\setminus K$. Hence, 
    $$|\mathcal{J}(K)|= \lambda_{\mathcal{J}}\binom{n-k}{m-k},$$
    where $\lambda_{\mathcal{J}}$ denotes the number of labelled copies of $\mathcal{J}$ into $\binom{[m]}{k}$ with one edge fixed. By the definition of $\mathcal{H}$, the degree of $K$ in $\mathcal{H}$ is 
    $$\deg_{\mathcal{H}}(K)=|\{\mathcal{J}'\in \mathcal{J}(K): \mathcal{J}' \text{ forms an edge of } \mathcal{H}\}|=\sum_{\mathcal{J}'\in \mathcal{J}(K)}Y_{\mathcal{J}'}, $$
    where $Y_{\mathcal{J}'}$ is the indicator of the event that $\mathcal{J}'$ is an edge of $\mathcal{H}$. Given $K\in V(\mathcal{H})$, the probability that a fixed $\mathcal{J}'\in\mathcal{J}(K)$ appears as an edge is
    $$\Pr[Y_{\mathcal{J}'}=1]=(1-\gamma)^{|\mathcal{J}|-1}\gamma^{\binom{m}{k}-|\mathcal{J}|}:=p.$$ 
    Therefore, the expected degree of $K$ is
    $$d:=\mathbb{E}[\deg_{\mathcal{H}}(K)]=\sum_{\mathcal{J}'\in \mathcal{J}(K)}\Pr[Y_{\mathcal{J}'}=1]
    =p\cdot \lambda_{\mathcal{J}}\binom{n-k}{m-k}=(c+o(1))n^{m-k},$$
    where $c=c(\mathcal{J},\gamma,m,k)$ is a constant independent of $n$.

    Observe that for a fixed $K\in V(\mathcal{H})$, the degree $\deg_{\mathcal{H}}(K)$ is a function of the independent random variables $\{X_T:T\in \binom{[n]}{k}\setminus\{K\}\}$, where $X_T$ is the indicator that $T\in\binom{[n]}{k}$ is colored blue.
    %, i.e., $$\deg_{\mathcal{H}}(K)=f\left(X_T:T\in \binom{[n]}{k}\Big\backslash\{K\}\right).$$
    We now apply Lemma~\ref{lem:BDI} to prove the concentration of $\deg_{\mathcal{H}}(K)$ around its expectation. For any $T\in \binom{[n]}{k}\setminus\{K\}$ with $|T\cap K|=i\in\{0,1,\ldots,k-1\}$, changing the color of such a $T$ affects $\deg_{\mathcal{H}}(K)$ by at most $b_T=O(n^{m-2k+i})$; indeed, an edge of $\mathcal{H}$ containing both $T$ and $K$ leaves at most $\binom{n-(2k-i)}{m-(2k-i)}$ choices for the remaining vertices. Therefore, 
    $$\sum_{T\in \binom{[n]}{k}\setminus\{K\}}b_T^2=\sum_{i=0}^{k-1}\sum_{T:~|T\cap K|=i}b_T^2=\sum_{i=0}^{k-1}O(n^{k-i+2(m-2k+i)})=O(n^{2m-2k-1}),$$ where the second equality follows from the fact that the number of $T$ with $|T\cap K|=i$ is at most $\binom{k}{i}\binom{n-k}{k-i}=O(n^{k-i})$.   

    Applying Lemma~\ref{lem:BDI} to $\deg_{\mathcal{H}}(K)$, we have
    %$\deg_{\mathcal{H}}(K)=f\left(X_T:T\in \binom{[n]}{k}\setminus\{K\}\right)$
    \begin{align*}
        \Pr[\deg_{\mathcal{H}}(K)\neq(1\pm d^{-\gamma})d]&=\Pr[|\deg_{\mathcal{H}}(K)-d|\ge d^{1-\gamma}]\le2\exp{\left(-\frac{2d^{2(1-\gamma)}}{\sum_{T\in \binom{[n]}{k}\setminus\{K\}}b_T^2}\right)}\\
        &=2\exp{\left(-\Theta\left(\frac{n^{2(1-\gamma)(m-k)}}{n^{2(m-k)-1}}\right)\right)}=2\exp{\left(-\Theta\left(n^{1-2\gamma(m-k)}\right)\right)}.
    \end{align*}
    Hence,
    \begin{align*}
        \Pr\left[\forall K\in V(\mathcal{H}), \deg_{\mathcal{H}}(K)=(1\pm d^{-\gamma})d\right]
        &=1-\Pr\left[\exists K\in V(\mathcal{H}), \deg_{\mathcal{H}}(K)\ne(1\pm d^{-\gamma})d\right]\\
        %&\ge 1-\sum_{K\in V(\mathcal{H})}\Pr[\deg_{\mathcal{H}}(K)\ne(1\pm d^{-\gamma})d]\\
        \ge1-\binom{n}{k}\cdot 2\exp{\left(-\Theta\left(n^{1-2\gamma(m-k)}\right)\right)}
        %\Pr[\deg_{\mathcal{H}}(K)\ne(1\pm d^{-\epsilon})d]\\
        &=1-2\exp{\left(O(\log n)-\Theta\left(n^{1-2\gamma(m-k)}\right)\right)}=1-o(1),
    \end{align*}
    because $\gamma<1/(2(m-k))$. In other words, with high probability every vertex in $\mathcal{H}$ is contained in $(1\pm d^{-\gamma})d$ edges, completing the proof of the claim. 
    \end{proof}
\end{document}